\documentclass{article}
\usepackage[multiple]{footmisc}
\usepackage[utf8]{inputenc}
\usepackage[T1]{fontenc}
\usepackage[left=1.8cm,right=1.8cm,top=1.8cm,bottom=1.8cm]{geometry}
\usepackage{tikz}
\usepackage{enumerate,mathtools}
\usepackage{amsmath}
\usepackage{amssymb}
\usepackage{amsthm,dsfont}
\newtheorem{thm}{Theorem}[section]
\newtheorem{lem}[thm]{Lemma}
\theoremstyle{remark}

\numberwithin{equation}{section}
\newcommand\numberthis{\addtocounter{equation}{1}\tag{\theequation}}
\title{On the limiting law of the length of the longest common and increasing subsequences in random words with arbitrary distributions}
\author{Clément Deslandes\footnote{C.M.A.P. Ecole Polytechnique, Palaiseau, 91120, France \& Georgia Institute of Technology, Atlanta, GA, 30332, USA (\texttt{clement.deslandes@poytechnique.edu}).}\qquad Christian Houdré\footnote{School of Mathematics, Georgia Institute of Technology, Atlanta, GA, 30332, USA (\texttt{houdre@math.gatech.edu}).} \footnote{Research supported in part by the 
grant $\sharp 524678$ from the Simons Foundation.
\newline\indent
Keywords:  Random Words, Longest Common Subsequences, Longest Increasing Subsequences, Weak Convergence, 
Optimal Alignment, Last Passage Percolation, Random Matrices.
\newline\indent
MSC 2010: 05A05, 60C05, 60F05.}}

\newcommand\ens[1]{\{1,\dots,#1\}}
\begin{document}
\maketitle

\begin{abstract}
Let $(X_k)_{k\geq 1}$ and $(Y_k)_{k\geq 1}$ 
be two independent sequences of i.i.d.~random variables, with values in a finite and totally 
ordered alphabet $\mathcal{A}_m:=\{1,\dots,m\}$, $m\ge 2$, having respective probability mass function  
$p^X_1,\dots,p^X_m$ and $p^Y_1,\dots,p^Y_m$. 
Let $LCI_n$ be the length of the longest common and weakly increasing subsequences 
in $X_1,...,X_n$ and $Y_1,...,Y_n$. Once properly centered and normalized, $LCI_n$ is shown 
to have a limiting distribution which is expressed as a functional of two independent multidimensional Brownian motions.
\end{abstract}

\section{Introduction and preliminary results}

\subsection{Introduction}

We analyze the asymptotic behavior of $LCI_n$, the length of the longest common subsequences in random words with an additional weakly increasing requirement. Throughout,  $(X_k)_{k\geq 1}$ and 
$(Y_k)_{k\geq 1}$ are two 
independent sequences of i.i.d.~random variables with values in the finite totally ordered 
alphabet $\mathcal{A}_m:=\{1,\dots,m\}$, $m\ge 2$, and respective pmf $p^X_1,\dots,p^X_m$, $p^X_i>0$, $i=1,\dots,m$ 
and $p^Y_1,\dots,p^Y_m$, $p^Y_i>0$, $i=1,\dots,m$.   
Next, $LCI_n$, the length of the longest common and weakly increasing subsequences of the two random words 
$X_1\cdots X_n$ and $Y_1\cdots Y_n$, is the largest integer $r\in\ens{n}$ such that there exist $1\leq i_1<\dots<i_{r}\leq n$ and $1\leq j_1<\dots<j_{r}\leq n$ such that
\begin{itemize}
\item $\forall s \in \ens{r}$, $X_{i_s}=Y_{j_s}$,
\item $X_{i_1}\leq X_{i_2}\leq \dots\leq X_{i_{r}}$ and $Y_{j_1}\leq Y_{j_2}\leq \dots\leq Y_{j_{r}}$,
\end{itemize}
and if no integer satisfies these two conditions, we set $LCI_n=0$.

A thorough discussion of the study of $LCI_n$, with potential applications, and a more complete bibliography, 
is present in \cite{breton2017limiting}, where the following is further proved (below, as usual, $\wedge$ is short for minimum):
\begin{thm}\label{unfthm} Let $X_k$ and $Y_k$ ($k=1,2,\dots$) be uniformly distributed over $\ens{m}$. Then,
\begin{multline}
 \frac{LCI_n-n/m}{\sqrt{n/m}}\xRightarrow[n \to \infty]{}\max_{0=t_0\leq t_1\leq\dots\leq t_m=1}  \Bigg[\left(-\frac{1}{m}\sum_{i=1}^m B^X_i(1)+\sum_{i=1}^m \left(B^X_i(t_i)-B^X_i(t_{i-1})\right)\right)\wedge \\ \left(-\frac{1}{m}\sum_{i=1}^m B^Y_i(1)+\sum_{i=1}^m \left(B^Y_i(t_i)-B^Y_i(t_{i-1})\right)\right)\Bigg],
 \end{multline}
where $B^X$ and $B^Y$ are two independent $m$-dimensional standard Brownian motions on $[0,1]$.
\end{thm}
The results of \cite{breton2017limiting} extended (and corrected) the proof of  
the case $m=2$ analyzed in \cite{HLM} and also conjectured the following generalization:
\begin{thm}\label{samedistthm}Let $X_k$ and $Y_k$ ($k=1,2,\dots$) have the same distribution, let $p_{\max}=\max_{i\in\ens{m}} p^X_i$ and let $k^*$ be its multiplicity. Then
\begin{multline}
 \frac{LCI_n-np_{\max}}{\sqrt{np_{\max}}}\xRightarrow[n \to \infty]{}\max_{0=t_0\leq t_1\leq\dots\leq t_{k^*}=1} \Bigg[\left(\frac{\sqrt{1-k^*p_{\max}}-1}{k^*}\sum_{i=1}^{k^*} B^X_i(1)+\sum_{i=1}^{k^*} \left(B^X_i(t_i)-B^X_i(t_{i-1})\right)\right)\wedge \\ \left(\frac{\sqrt{1-{k^*}p_{\max}}-1}{k^*}\sum_{i=1}^{k^*} B^Y_i(1)+\sum_{i=1}^{k^*} \left(B^Y_i(t_i)-B^Y_i(t_{i-1})\right)\right)\Bigg],
 \end{multline}
where $B^X$ and $B^Y$ are two independent $k^*$-dimensional standard Brownian motions on $[0,1]$.
\end{thm}

Clearly, in case $k^* = m$,   
the two limiting distributions in \eqref{unfthm} and \eqref{samedistthm} are the same 
but they differ otherwise.  Indeed, \eqref{unfthm} 
involves two independent $m$-dimensional Brownian motions while 
\eqref{samedistthm} involves $k^*$-dimensional ones.  
So, in particular, if $k^*=1$, then the right-hand side of (1.2) is just  
the minimum of two independent centered normal random variables.   In view of the results 
obtained in the one-sequence case, e.g., see \cite{HL1}, \cite{B-GH},  and the many references 
therein, it is tantalizing to conjecture that both the right-hand side of \eqref{unfthm} 
and of \eqref{samedistthm} can be realized as maximal eigenvalues of some 
Gaussian random matrix models.

Below, we aim to obtain the limiting distribution of $LCI_n$, without 
assuming that the $X_k$ and $Y_k$ ($k=1,2,\dots$) have the same distribution; providing also an alternative proof of 
Theorem~\ref{unfthm} as well as a proof of the conjectured  \eqref{samedistthm}.  
A brief description of the content of our notes is as follows: the rest of the current section is devoted to studying 
the asymptotic mean of $LCI_n$.  This asymptotic mean result is already not so predictable and allows for the proper centering in the  limiting theorem 
whose proof is provided in the next section.  The third and final section 
is mainly devoted to studying extensions and complements, 
such as results for sequences with blocks and infinite countable alphabets.  

\noindent
{\bf Acknowledgements:}  We sincerely thank an Associate Editor and a referee for their 
detailed readings and numerous comments which 
greatly helped to improve this manuscript.

\subsection{Probability}

For $i\in \ens{m}$ and $j\in\ens{n}$, let $\ell\in\mathbb{N}=\{0,1,2,\dots\}$ be such that $j+\ell\leq n+1$, and let \begin{equation*}
N^{X,i}_{j,\ell}=\sum_{k=0}^{\ell-1}\mathds{1}_{X_{j+k}=i}\qquad \left(\text{resp.}\:N^{Y,i}_{j,\ell}=\sum_{k=0}^{\ell-1}\mathds{1}_{Y_{j+k}=i}\right),
\end{equation*} 
be simply the number of letters $i$ between,  and including, $j$ and $j+\ell-1$ in $X_1,...,X_n$ (resp. $Y_1,...,Y_n$),  with the convention that the sum is zero in case $\ell=0$. 
From the very definition of $LCI_n$, it is clear that 
\begin{equation*}
LCI_n=\max_{\substack {\ell^X,\ell^Y\in\mathbb{N}^{m} \\ \ell^X_1 + \dots +\ell^X_m=n \\ \ell^Y_1 + \dots +\ell^Y_m=n}} \bigg( N^{X,1}_{1,\ell^X_1}\wedge N^{Y,1}_{1,\ell^Y_1}+N^{X,2}_{\ell^X_1,\ell^X_2}\wedge N^{Y,2}_{\ell^Y_1,\ell^Y_2}+\dots+N^{X,m}_{\ell^X_1+\dots+\ell^X_{m-1},\ell^X_m}\wedge N^{Y,m}_{\ell^Y_1+\dots+\ell^Y_{m-1},\ell^Y_m}\bigg).
\end{equation*}

Next, let $\Lambda=\{\lambda\in\left(\mathbb{R}_+\right)^{m}=[0,+\infty)^{m}\::\: \lambda_1+\dots+\lambda_m=1\}$. For $\lambda\in\Lambda$, let 
\begin{equation}
\ell^n(\lambda)_i=\lfloor (\lambda_1+\dots+\lambda_i)n\rfloor-\lfloor (\lambda_1+\dots+\lambda_{i-1})n\rfloor,
\end{equation}
where $\lfloor . \rfloor$ is the usual integer part, aka the floor, function. When $\lambda$ runs through $\Lambda$, $\ell^n(\lambda)=(\ell^n(\lambda)_1, \dots, \ell^n(\lambda)_m)$ runs exactly 
through $\left\{\ell\in\mathbb{N}^m \::\: \ell_1+\dots+\ell_m=n\right\}$, so

\begin{multline}\label{introlambda}
LCI_n=\max_{\substack {\lambda^X,\lambda^Y\in \Lambda}}\bigg( N^{X,1}_{1,\ell^n(\lambda^X)_1}\wedge N^{Y,1}_{1,\ell^n(\lambda^Y)_1}+N^{X,2}_{\ell^n(\lambda^X)_1,\ell^n(\lambda^X)_2}\wedge N^{Y,2}_{\ell^n(\lambda^Y)_1,\ell^n(\lambda^Y)_2}+\dots\\ +N^{X,m}_{\ell^n(\lambda^X)_1+\dots+\ell^n(\lambda^X)_{m-1},\ell^n(\lambda^X)_m}\wedge N^{Y,m}_{\ell^n(\lambda^Y)_1+\dots+\ell^n(\lambda^Y)_{m-1},\ell^n(\lambda^Y)_m}\bigg).
\end{multline}

For ease of notations, throughout the paper, for all $x\in\left(\mathbb{R}^m\right)^2$, we 
write $x=(x^X,x^Y)$ so, for example, above, $\lambda^X,\lambda^Y\in \Lambda$ becomes $\lambda\in\Lambda^2$.

For $i\in\ens{m}$ and $t\in [0,1]$, let now \begin{equation}\label{defB}
\widetilde{B}^{n,X}_i(t)=\frac{N^{X,i}_{1,\lfloor tn \rfloor}-p^X_i t n}{\sqrt{p^X_i(1-p^X_i)n}},\qquad \left(\text{resp.}\: \widetilde{B}^{n,Y}_i(t)=\frac{N^{Y,i}_{1,\lfloor tn \rfloor}-p^Y_i t n}{\sqrt{p^Y_i(1-p^Y_i)n}}\right),
\end{equation} and for $\lambda\in\Lambda^2$, let \begin{align}\label{defVX}
\widetilde{V}^{n,X}_i(\lambda^X)&=\sqrt{p^X_i(1-p^X_i)}\left(\widetilde{B}^{n,X}_i(\lambda^X_1+\dots+\lambda^X_i)-\widetilde{B}^{n,X}_i(\lambda^X_1+\dots+\lambda^X_{i-1})\right),\\\label{defVY}
\widetilde{V}^{n,Y}_i(\lambda^Y)&=\sqrt{p^Y_i(1-p^Y_i)}\left(\widetilde{B}^{n,Y}_i(\lambda^Y_1+\dots+\lambda^Y_i)-\widetilde{B}^{n,Y}_i(\lambda^Y_1+\dots+\lambda^Y_{i-1})\right),
\end{align} so that (\ref{introlambda}) becomes \begin{equation}\label{exprlcin}
LCI_n=\max_{\substack {\lambda\in \Lambda^2}} \sum_{i=1}^m\Bigg[\left(n p^X_i\lambda^X_i+\sqrt{n}\widetilde{V}^{n,X}_i(\lambda^X)\right)\wedge \left(n p^Y_i\lambda^Y_i+\sqrt{n}\widetilde{V}^{n,Y}_i(\lambda^Y)\right)\Bigg].
\end{equation}

The above identity provides a representation of $LCI_n$ as a maximum over the locations, $\lambda\in\Lambda^2$, where to pick in each word $X_1,\dots,X_n$ and $Y_1,\dots,Y_n$, the letters $1,2,\dots,m$ in order to form a common sub-word. This is  different from the approach in \cite{breton2017limiting}, where the maximum is over 
the numbers of letters $1,2,\dots,m$ in a common sub-word. Of course the two representations are equivalent. However, the advantage of our approach is that $\lambda$ takes its values in a deterministic set, as opposed to a random set.

In order to keep dealing with maxima it will be convenient to replace 
$\widetilde{B}^n_i$ in (\ref{defB}) by its continuous alternative: for $i\in\ens{m}$ and $t\in [0,1]$, 
let \begin{equation*}
B^{n,X}_i(t)=\frac{N^{X,i}_{1,\lfloor tn \rfloor}+(tn-\lfloor tn \rfloor)\mathds{1}_{X_{\lfloor tn \rfloor+1}=i} -p^X_i t n}{\sqrt{p^X_i(1-p^X_i)n}},\qquad \left(\text{resp.}\:B^{n,Y}_i(t)=\frac{N^{Y,i}_{1,\lfloor tn \rfloor}+(tn-\lfloor tn \rfloor)\mathds{1}_{Y_{\lfloor tn \rfloor+1}=i} -p^Y_i t n}{\sqrt{p^Y_i(1-p^Y_i)n}}\right).
\end{equation*} 
Next define $V^{n,X}$, $V^{n,Y}$ just as in (\ref{defVX}) and (\ref{defVY}), 
replacing $\widetilde{B}$ by $B$, and let\begin{equation*}
LCI^c_n=\max_{\substack {\lambda\in \Lambda^2}} \sum_{i=1}^m\Bigg[\left(n p^X_i \lambda^X_i+\sqrt{n}V^{n,X}_i(\lambda)\right)\wedge \left(n p^Y_i \lambda^Y+\sqrt{n}V^{n,Y}_i(\lambda)\right)\Bigg].
\end{equation*}

Our analysis rests upon estimating the variations of $B^{n,X}_i$ and of $B^{n,Y}_i$.  
To do so, let $\eta\in(0,1/6)$ and let $A_n^\eta$ be the event: 
\begin{equation*}
\forall i\in\ens{m}, \forall j \in \ens{n}, \forall \ell\in\{0,\dots,n+1-j\}, \left\vert \frac{N^{X,i}_{j,\ell}-p^X_i \ell}{\sqrt{n}}\right\vert\leq \frac{n^\eta}{2}\sqrt{\frac{\ell}{n}}\: \text{and}\:\left\vert \frac{N^{Y,i}_{j,\ell}-p^Y_i \ell}{\sqrt{n}}\right\vert\leq \frac{n^\eta}{2}\sqrt{\frac{\ell}{n}}.
\end{equation*}
By Hoeffding's  inequality, 
\begin{equation}\label{azuma}
1-\mathds{P}\left(A_n^\eta\right)\leq 2n(n+1)m \exp\left(-\frac{n^{2\eta}}{2}\right),
\end{equation}and so if $A_n^\eta$ occurs, then for all $x,y$ in $[0,1]$ and $i\in\ens{m}$, \begin{equation*}
\left\vert \sqrt{p^X_i(1-p^X_i)}\left(B^{n,X}_i(y)-B^{n,X}_i(x)\right)\right\vert \leq \frac{n^\eta}{2}\sqrt{\vert y-x\vert +\frac{1}{n}}.
\end{equation*}
and in particular,
\begin{equation*}
\left\vert \sqrt{p^X_i(1-p^X_i)}\left(B^{n,X}_i(y)-B^{n,X}_i(x)\right)\right\vert 
\leq \frac{n^\eta}{2}\sqrt{\vert y-x\vert}+\frac{n^{\eta-1/2}}{2}\leq n^\eta,
\end{equation*} 
and the same applies to $Y$ instead of $X$.

\subsection{Asymptotic mean: distinct cases}\label{secamean}

Let us investigate the limiting behavior of $LCI_n/n$. From (\ref{exprlcin}),
\begin{equation*}
\frac{LCI_n}{n}=\max_{\substack {\lambda\in \Lambda^2}} \sum_{i=1}^m\left[\left(p^X_i\lambda^X_i+\frac{\widetilde{V}^{n,X}_i(\lambda^X)}{\sqrt{n}}\right)\wedge \left(p^Y_i\lambda^Y_i+\frac{\widetilde{V}^{n,Y}_i(\lambda^Y)}{\sqrt{n}}\right)\right].
\end{equation*}
Note that $\vert \widetilde{V}^{n,X}_i(\lambda^X)-V^{n,X}_i(\lambda^X)\vert\leq 1/\sqrt{n}$ (and similarly for $Y$). Thus, using (throughout the paper) the following elementary inequality, valid for any $a,b,c,d\in\mathbb{R}$,
\begin{equation}\label{elem}
\left\vert a\wedge b - (a+c)\wedge (b+d)\right\vert \leq \max(\vert c\vert , \vert d\vert),
\end{equation}we get 
\begin{equation}\label{closetoc}
\left| \frac{LCI_n}{n}-\frac{LCI^c_n}{n}\right| \leq \frac{m}{n}.
\end{equation}

Moreover, if $A^{\eta}_n$ occurs, then for all $\lambda\in\Lambda^2$, 
\begin{equation*}
\left| \sum_{i=1}^m \left[\left(p^X_i\lambda^X_i+\frac{V^{n,X}_i(\lambda^X)}{\sqrt{n}}\right)\wedge \left(p^Y_i\lambda^Y_i+\frac{V^{n,Y}_i(\lambda^Y)}{\sqrt{n}}\right)\right] - \sum_{i=1}^m \left[\left(p^X_i\lambda^X_i\right)\wedge \left(p^Y_i\lambda^Y_i\right)\right]\right| \leq m n^{\eta-1/2},
\end{equation*} 
so, letting $f:\left(\mathbb{R}^{m}\right)^2\rightarrow \mathbb{R}$ be given via 
\begin{equation}\label{defef}
f:(y^X,y^Y)\mapsto \sum_{i=1}^m \left[\left(p^X_i y^X_i\right)\wedge \left(p^Y_i y^Y_i\right)\right],
\end{equation}
we have:  
\begin{equation*}
\left| \frac{LCI_n}{n} - \max_{\substack{\lambda\in\Lambda^2}}f(\lambda)\right| \leq {m}{n^{\eta-1/2}}.
\end{equation*}

By the Borel-Cantelli lemma (recalling \eqref{azuma}), almost surely, eventually 
$A^{\eta}_n$ occurs so $LCI^c_n/n$ and $LCI_n/n$ both converge almost 
surely to 
\begin{equation}\label{defemax}
e_{\max}:=\max_{\substack{\lambda\in\Lambda^2}}f(\lambda).
\end{equation}
From
\begin{equation*}
\frac{LCI_n}{n}\xrightarrow[n \to \infty]{}e_{max}, \text{a.s.},
\end{equation*}
we also get by dominated convergence
\begin{equation*}
\frac{\mathbb{E}LCI_n}{n}\xrightarrow[n \to \infty]{}e_{max}.
\end{equation*}

One can think of $e_{\max}$ as the length ratio of the longest common and 
increasing subsequences in a continuous, non-probabilistic setup: 
the letters have density masses $p^X_1, p^X_2, \dots, p^X_m$ and 
$p^Y_1, p^Y_2, \dots, p^Y_m$.

Now, let 
$$U=\left\{u\in(\mathbb{R}_+)^m\::\: \frac{u_1}{p^X_1}
+\dots+\frac{u_m}{p^X_m}\leq 1, \frac{u_1}{p^Y_1}+\dots+\frac{u_m}{p^Y_m}\leq 1\right\},$$ 
and let $\phi:\mathbb{R}^{m}\rightarrow \mathbb{R}$  be given by
$\phi:u\mapsto u_1+\dots+u_m$.

On $U$, there is a correspondence between $f$ in \eqref{defef},
and the above $\phi$.  Indeed, for $\lambda\in\Lambda^2$, defining $u$ by 
$u_i=\left(p^X_i \lambda^X_i\right)\wedge \left(p^Y_i \lambda^Y_i\right)$, $f(\lambda)=\phi(u)$, 
and for $u\in U$, there exists $\lambda\in \Lambda^2$,  
such that $\lambda^X_i\ge {u_i/p^X_i}$ and $\lambda^Y_i\ge {u_i/p^Y_i}$ 
so that $f(\lambda)\geq \phi(u)$. 
Therefore, $e_{\max}=\max_{\substack{u\in U}}\phi(u)$. 
Also, let 
\begin{equation}\label{KandL}
K_{\Lambda^2}=f^{-1}\left(\{e_{\max}\}\right)\cap \Lambda^2, \  
{\rm and}\  L_U=\phi^{-1}\left(\{e_{\max}\}\right)\cap U.
\end{equation}
The above correspondence provides for each element of $K_{\Lambda^2}$ an  
element of $L_U$, and for each element of $L_U$ at least one element of 
$K_{\Lambda^2}$ (if one of the two inequalities defining $U$ is strict, then there is 
more than one way to define the corresponding $\lambda$). 
Next, let $I$ be the set of integers $i\in\{1,\dots,m\}$ such that there exists $u^i\in L_U$ 
with $u^i_i>0$. One can think of $I$ as the letters that can be used to 
maximize $\phi$, or, equivalently, to maximize $f$. Let \begin{equation}\label{defui}
u^I=\frac{1}{\vert I \vert}\sum_{i\in I} u^i,
\end{equation}so $u^I\in L_U$ and for all $i\in I$, $u^I_i>0$. 
Thanks to the above correspondence, we define (and will use throughout 
the paper) $a\in \Lambda^2$ such that $a^X_i=a^Y_i=0$ for all $i\notin I$ and $a^X_i\geq u^I_i/p^X_i$, $a^Y_i\geq u^I_i/p^Y_i$, 
for all $i\in I$ ($a$ is a correspondent of $u_I$). Since $f(a)\geq \phi(u^I)=e_{\max}$, $a\in K_{\Lambda^2}$. We shall see, and use, that when restricting the alphabet to $I$, asymptotically (when properly centered and normalized) the distribution of $LCI_n$ remains unchanged.

Two distinct cases need to be analyzed in order to study the limiting distribution of $LCI_n$.

\paragraph{Case a)} There exists $u\in L_U$ such that $\frac{u_1}{p^X_1}
+\dots+\frac{u_m}{p^X_m} =1$ and 
$\frac{u_1}{p^Y_1}+\dots+\frac{u_m}{p^Y_m}<1$.

For example, when $p^X=(3/8,3/8,1/4)$ and $p^Y=(1/2,3/8,1/8)$. 
Here the maximum is $3/8$, and $I=\{1,2\}$.

Heuristically, this case indicates that the length of the common words is limited by the word 
$X_1\cdots X_n$ and not by $Y_1\cdots Y_n$. 
Using the correspondence between $L_U$ and $K_{\Lambda^2}$, this case is equivalent to 
the following statement: there exists $\lambda\in K_{\Lambda^2}$ such that for all $i\in\ens{m}, p^X_i\lambda^X_i\leq p^Y_i\lambda^Y_i$ with at least one strict inequality. In this case, one has:  
\begin{lem}\label{icasa}
Let $p^X_{\max}=\max_{i\in\ens{m}} p^X_i$. Then $I=\{i\in\ens{m}\::\: p^X_i=p^X_{\max}\}$ and $e_{\max}=p^X_{\max}$. Moreover there exists $i_1\in I$ such that $p^Y_{i_1}>p^X_{\max}$.
\end{lem}

\begin{proof}
Let $i,j\in\ens{m}$ be such that $p^X_i<p^X_j$, and assume, by contradiction, that $i\in I$. 
Let $u\in L_U$ satisfying $\frac{u_1}{p^X_1}+\dots+\frac{u_m}{p^X_m} =1$ and $\frac{u_1}{p^Y_1}+\dots+\frac{u_m}{p^Y_m}<1$, and let $v=(u^i+u)/2$, so that $v\in U$, $v_i>0$, 
$\frac{v_1}{p^X_1}+\dots+\frac{v_m}{p^X_m} \leq 1$ and $\frac{v_1}{p^Y_1}+\dots+\frac{v_m}{p^Y_m}<1$. 
Let, for $\varepsilon>0$, $v({\varepsilon})$ be the vector $v$ except at the 
coordinates $i$ and $j$ where  $v({\varepsilon})_i:=v_i-\varepsilon p^X_i$ 
and $v({\varepsilon})_j:=v_j+\varepsilon p^X_j$.  
It is clear that, when $\varepsilon$ is small enough, 
$v(\varepsilon)\in U$ and $\phi\left(v(\varepsilon)\right)
=e_{\max}+\varepsilon(p^X_j-p^X_i)>e_{\max}$, 
leading to a contradiction.  
Hence $I\subset\{i\in\ens{m}\::\: p^X_i=p^X_{\max}\}$. 
Reciprocally, let $i\in\ens{m}$ be such that $p^X_i=p^X_{\max}$ and let $j\in I$. 
If $i=j$ we are done. Otherwise, one can slightly change $u$ by adding $\varepsilon$ to the $ith$ coordinate and subtracting $\varepsilon$ to the $jth$ coordinate so that $\phi(u)$ remains unchanged, and $u$ is still in $U$ (for $\varepsilon$ small enough), so $I=\{i\in\ens{m}\::\: p^X_i=p^X_{\max}\}$.

Since $\frac{u_1}{p^X_1}+\dots+\frac{u_m}{p^X_m}=\sum_{i\in I}\frac{u_i}{p^X_{\max}}>\sum_{i\in I} \frac{u_i}{p^Y_i}$, there exists $i_1\in I$ such that $p^Y_{i_1}>p^X_{\max}$. 
It is finally clear that $e_{\max}=p^X_{\max}$, completing the proof.
\end{proof}

As a consequence of the above lemma, we prove next that 
\begin{equation}\label{defJ}
J:=\left\{\lambda^X\in\Lambda\::\: \forall i\notin I, \ \lambda^X_i=0, \ \sum_{i\in I} \frac{\lambda^X_i}{p^Y_i}\leq \frac{1}{p^X_{\max}}\right\} = \left\{\lambda^X\::\:\lambda\in K_{\Lambda^2}\right\},
\end{equation} 
(in particular, this set is non-empty which is all that is really needed in the rest of the proof).  
To show this equality, 
first note that $\left\{\lambda^X\::\:\lambda\in K_{\Lambda^2}\right\}\subset J$ since, indeed, 
when $\lambda\in  K_{\Lambda^2}$, for every $i\in I$, $p^X_{\max}\lambda^X_i\leq p^Y_i \lambda^Y_i$ and then 
take the sum. 
Conversely, if $\lambda^X\in J$, $\sum_{i\in I} {p^X_{\max}\lambda^X_i}/{p^Y_i}\leq 1$, 
so let $\lambda^Y$ be such that 
for every $i\in I$, $\lambda^Y_i\geq {p^X_{\max}\lambda^X_i}/{p^Y_i}$ and $\sum_{i\in I} \lambda^Y_i=1$, while for $i\in I^c$, 
let $\lambda^Y_i=0$.  Clearly, $\lambda\in K_{\Lambda^2}$.

\paragraph{Case b)} For all $u\in L_U$, $\frac{u_1}{p^X_1}+\dots+\frac{u_m}{p^X_m} = \frac{u_1}{p^Y_1}+\dots+\frac{u_m}{p^Y_m}=1$.

Heuristically, this second case indicates that in order to form the longest common words, it is 
necessary to make full use of both words. Using the correspondence between $L_U$ and $K_{\Lambda^2}$, this case is equivalent to the following: for all $\lambda\in K_{\Lambda^2}$, 
for all $i\in\ens{m}, p^X_i\lambda^X_i= p^Y_i\lambda^Y_i$. We can further distinguish two 
subcases, namely, we are in Case b1)  if each coordinate of $P^X:=\left(1/p^X_i\right)_{i\in I}\in \mathbb{R}^{I}$ is equal to each coordinate of 
$P^Y=\left(1/p^Y_i\right)_{i\in I}\in \mathbb{R}^{I}$, and in Case b2) otherwise. 

For example, if $p^X=(1/3,1/3,2/9,1/9)$ and $p^Y=(1/3,1/3,1/9,2/9)$, we are in Case b1) and $e_{\max}=1/3$. If $p^X=(2/3,1/6,1/6)$ and $p^Y=(1/6,2/3,1/6)$, we are in 
Case b2) and $e_{\max}=4/15$. In both of these examples, $I=\{1,2\}$.

Below $\text{Span}(P^X)$ (resp. $\text{Span}(P^Y)$) is the linear span of $P^X$ (resp. $P^Y$).

\begin{lem}\label{indep} In Case b2), there exists a unique pair of reals $s,t$ such that  
$sP^X+tP^Y=(1)_{i\in I}$\end{lem}

\begin{proof}
The only alternatives to Case b1) are: $P^X$ and $P^Y$ are linearly independent, or $P^X$ and $P^Y$ are linearly dependent and $P^X\neq P^Y$. If the latter, given that $P^X$ and $P^Y$ have positive coordinates, $P^X<P^Y$ (coordinate by coordinate) or $P^Y<P^X$. But $P^X<P^Y$ clearly implies that Case a) occurs, and not Case b) leading to a contradiction (and similarly $P^Y<P^X$). 
Therefore, the only alternative to Case b1) is for $P^X$ and $P^Y$ to be linearly independent. We now prove that $H:=(1)_{i\in I}\in\text{Span}(P^X,P^Y)$. To do so, we use an 
elementary duality result: if $E$ is a finite-dimensional space with dual 
$E^*$, and if $l_1,l_2,l_3\in E^*$, then $\text{Ker}(l_1)\cap\text{Ker}(l_2)\subset\text{Ker}(l_3)$ if and only if  $l_3\in\text{Span}(l_1,l_2)$. Indeed, considering the restrictions $l_{2|\text{Ker}(l_1)}$ and 
$l_{3|\text{Ker}(l_1)}$ of $l_2$ and $l_3$ to the subspace $\text{Ker}(l_1)$, we have $\text{Ker}(l_{2|\text{Ker}(l_1)})\subset\text{Ker}(l_{3|\text{Ker}(l_1)})$.  Therefore,  $l_{3|\text{Ker}(l_1)}=\lambda l_{2|\text{Ker}(l_1)}$ for 
some $\lambda\in \mathbb{R}$, and if $u\notin \text{Ker}(l_1)$, then $l_3=\lambda l_2 +\frac{l_3(u)-\lambda l_2(u)}{l_1(u)}l_1$ 
(because this is true on $\text{Ker}(l_1)$ and on $u$). 
So, returning to our problem, $H\in\text{Span}(P^X,P^Y)$ is equivalent to: $\text{Ker}(P^{X^*})\cap\text{Ker}(P^{Y^*})\subset\text{Ker}(H^*)$, where for any $L\in\mathbb{R}^I$, $L^*$ denotes the linear form defined by $L^*(y)=L\cdot y$. Let $x\in\text{Ker}({(P^X)}^*)\cap\text{Ker}({(P^Y)}^*)$. Clearly, there exists $\varepsilon>0$ such that $u^I+\varepsilon x$ and $u^I-\varepsilon x$ have non-negative coordinates, 
and so they are in $L_U$, and 
$H^*(u^I+\varepsilon x)=H^*(u^I-\varepsilon x)=e_{\max}$ otherwise one of them 
would be greater than $e_{\max}$, hence $x\in \text{Ker}(H^*)$.
\end{proof}

For instance, taking again 
$p^X=(2/3,1/6,1/6)$ and $p^Y=(1/6,2/3,1/6)$, 
we get $P^X=(3/2,6), P^Y=(6,3/2)$ and $s=t=2/15$.

Without loss of generality (switching the roles of $X$ and $Y$), one can thus assume that either 
Case a) or Case b) occurs.

In Case b), the following technical lemma, whose proof (given in the Appendix) is not crucial to 
understand the rest of this manuscript, is needed to state our main theorem.  
Let us define first, in Case b1),
\begin{equation}
s_X:=\begin{cases} 
  \max_{{i\in I^c:p^X_i\geq e_{\max}}} \frac{p^Y_i(p^X_i-e_{\max})}{e_{\max}(p^X_i-p^Y_i)} & 
  \text{if }\{i\in I^c, p^X_i\geq e_{\max}\}\neq \emptyset, \\
      0, & \text{if }\{i\in I^c, p^X_i\geq e_{\max}\}= \emptyset, \\
      \end{cases}
\quad t_X:=1-s_X,
\end{equation}
and, similarly, 
 \begin{equation}
s_Y:=\begin{cases} 
  \max_{{i\in I^c, p^Y_i\geq e_{\max}}} \frac{p^X_i(p^Y_i-e_{\max})}{e_{\max}(p^Y_i-p^X_i)},  
  & \text{if } \{i\in I^c:p^Y_i\geq e_{\max}\}\neq \emptyset, \\
      0, & \text{if }\{i\in I^c, p^Y_i\geq e_{\max}\}= \emptyset, \\
      \end{cases}
\quad t_Y:=1-s_Y.
\end{equation}

It is clear, from the definition of $I$, that if $i\in I$ is such that $p^X_i\geq e_{\max}$, 
then $p^Y_i<e_{\max}$, therefore $s_X$ and $s_Y$ are well defined 
and one can check that $s_X, t_X, s_Y, t_Y\in [0,1]$.   

In order to state our next lemma, below let 
$E=\{x\in\mathbb{R}^{m}\::\: x_1+\dots+x_m=0\}$ and let 
$E'=\left\{x\in E : \forall i\in I^c, x_i\geq 0 \right\}$.

\begin{lem}\label{exprm} Let $\nu\in\left(\mathbb{R}^m\right)^2$ be such 
that for all $i\in I^c, \nu^X_i=\nu^Y_i=0$, then the following maximum is well defined:
\begin{equation}\label{defm}
\mathfrak{m}(\nu):=\max_{x\in E'^2} \sum_{i=1}^m \left[\left(p^X_i x^X_i
+\nu^X_i\right)\wedge \left(p^Y_i x^Y_i+\nu^Y_i\right)\right],  
\end{equation}  
and 
\begin{equation}\label{local}
\mathfrak{m}(\nu)=\max_{\substack{x\in E'^2\\ \|x\|_{\infty}\leq 
2Cm \|\nu\|_{\infty}}} \sum_{i=1}^m \left[\left(p^X_i x^X_i
+\nu^X_i\right)\wedge \left(p^Y_i x^Y_i+\nu^Y_i\right)\right], 
\end{equation}
for some constant $C>0$, depending only on 
$p^X$ and $p^Y$, as given in Lemma~\ref{proj}.  
In Case b1), writing $S^{\bullet}:=\sum_{i\in I}\nu^{\bullet}_i$, then 
\begin{equation}
\mathfrak{m}(\nu)= \begin{cases} 
      s_X S^Y+t_X S^X, & if S^X\leq S^Y, \\
      s_Y S^X+t_Y S^Y, & if S^X\geq S^Y. \\
   \end{cases}.
\end{equation}
In Case b2), and recalling the notations of Lemma~\ref{indep}, then 
\begin{equation}
\mathfrak{m}(\nu)=\sum_{i\in I} \left(\frac{s}{p^X_i}\nu^X_i+ \frac{t}{p^Y_i}\nu^Y_i\right).
\end{equation}
\end{lem}

\subsection{Representation of $e_{\max}$}

We now aim to give a more explicit expression for $e_{max}$ defined by \eqref{defemax}.  
To do so, let us start with the following lemma which asserts that, in the non-probabilistic setup, 
"two letters are enough to reach the maximum".

\begin{lem}\label{reductotwo}There exist $i,j\in\ens{m}$ and $\lambda\in K_{\Lambda^2}$ 
such that for all $k\notin\{i,j\}, \lambda^X_k=\lambda^Y_k=0$.
\end{lem}

\begin{proof}
Let $u\in L_U$ having (at least) three non-zero coordinates. Then, 
recalling the correspondence between $L_U$ and $K_{\Lambda^2}$, in order 
to prove the result it is enough to show that there exists a 
$v\in L_U$ having one less null coordinate.  
Without loss of generality, let $u_1,u_2,u_3>0$, and let\begin{equation*}
V=\left\{x\in\mathbb{R}^m\::\:\sum_{i=1}^m \frac{x_i}{p^X_i}
=\sum_{i=1}^m \frac{x_i}{p^Y_i}=0,x_4=\dots=x_n=0\right\}.
\end{equation*}
Since the dimension of $V$ is at least one, let $x\in V\setminus \{0\}$.  
Then clearly,  there exists $t\in\mathbb{R}$ such that $v:=u+tx$ has 
non-negative coordinates and one more null coordinate than $u$. 
Moreover, $v\in L_U$, which completes the proof.
\end{proof}

If there exists $u\in L_U$ such all its coordinates except one, call it $i$, are zeros, 
then $e_{\max}=p^X_i\wedge p^Y_i$. Otherwise, let $i,j$ be defined as in the 
statement of the lemma.  At first, assume that $p^X_i=p^X_j$ and that $p^Y_i\leq p^Y_j$, 
then $e_{\max}\leq (\lambda^X_i p^X_i\wedge \lambda^Y_i p^Y_j)+(\lambda^X_j p^X_i\wedge \lambda^Y_j p^Y_j)\leq (\lambda^X_i p^X_i+\lambda^X_j p^X_i)\wedge( \lambda^Y_i p^Y_j+ \lambda^Y_j p^Y_j)=p^X_i\wedge p^Y_i$, so $e_{\max}=p^X_i\wedge p^Y_i$ and we are actually in the first case, giving 
a contradiction.  Similarly, if $p^X_i\leq p^X_j$ and $p^Y_i\leq p^Y_j$, using $\lambda^X_i p^X_i\wedge \lambda^Y_i p^Y_i\leq \lambda^X_i p^X_j\wedge \lambda^Y_i p^Y_j$ we get a contradiction as well. 
Therefore, in the second case, necessarily, possibly permuting $i$ and $j$, $p^X_i<p^X_j$ and $p^Y_i>p^Y_j$.  
Additionaly, it is necessary to have that $p^X_i< p^Y_i$, otherwise $e_{\max}=p^Y_i$ and we 
are in the first case. Similarly, $p^Y_j<p^X_j$. Then, in this case, the maximum is when 
the quantities in each minima are equal, and so one shows that \begin{equation*}
e_{\max}=e(i,j):=\frac{p^X_i p^Y_i(p^X_j-p^Y_j)+p^X_j p^Y_j(p^Y_i-p^X_i)}{p^Y_i p^X_j-p^X_i p^Y_j}.
\end{equation*}

Therefore,

\begin{equation}\label{finalemax}
e_{\max}=\max\Big(\max_{1\leq i\leq m} \left(p^X_i\wedge p^Y_i\right), \max_{\substack{i,j\::\: p^X_i<\,p^X_j\\\quad\rotatebox{90}{$\scriptstyle >$}\quad\: \rotatebox{90}{$\scriptstyle <$}\\ \quad\: p^Y_i>\,p^Y_j}}e(i,j)\Big).
\end{equation}

\noindent
Note that \begin{equation}
\max_{1\leq i\leq m} \left(p^X_i\wedge p^Y_i\right)\leq e_{\max}\leq \left(\max_{1\leq i\leq m}  p^X_i\right)\wedge\left(\max_{1\leq i\leq m}  p^Y_i\right), 
\end{equation}
\noindent
where the left inequality is clear, while the right one is easily seen  
from the expression of $f$.   Note also that above, $e_{max}$ 
is equal to the lower bound when the second max in \eqref{finalemax} is over the 
empty set, and is equal to the upper bound when there exists $i$ such that $p^X_{\max}=p^X_i\leq p^Y_i$ or $p^Y_{\max}=p^Y_i\leq p^X_i$.

When $p^X=p^Y$ (same distribution for the two words), we see that $e_{\max}=\max_{i\in\ens{m}}p^X_i$ is minimal when $p^X$ is uniform (for a given alphabet). This is to be contrasted 
with the case of the length of the longest common subsequences, $LC_n$ (defined just as $LCI_n$, 
but without the increasing condition).  Indeed, little is known about $\gamma^*:=\lim_{n\to +\infty} {\mathbb{E}LC_n}/{n}$, for instance whether or not it is 
minimal (for a given alphabet) for the uniform distribution.
Since $LC_n$ is defined with one less constraint than $LCI_n$, clearly $e_{\max}\leq \gamma^*$ 
which is of potential interest since the exact value of $\gamma^*$ is unknown, even in the 
binary uniform case.  (This last inequality provides a lower bound on $\gamma^*$, no matter the distributions on the letters.  For uniform letters, $e_{\max} = 1/m$, although it is known that, then, asymptotically, $\gamma^*\sim 2/\sqrt{m}$, see \cite{KLM}.)

\subsection{A criterion to distinguish the three cases}

For a given distribution, it is not completely apparent which 
situation is in play as far as 
the respective cases a), b1) and b2) are concerned.  
Our next result makes this more transparent.  First, set 
$$e_1=\max_{1\leq i\leq m} \left(p^X_i\wedge p^Y_i\right), \qquad  e_2=\max_{\substack{i,j\::\: p^X_i<\,p^X_j\\\quad\rotatebox{90}{$\scriptstyle >$}\quad\: \rotatebox{90}{$\scriptstyle <$}\\ 
\quad\: p^Y_i>\,p^Y_j}}e(i,j),$$  so that, by \eqref{finalemax}, $e_{\max}=\max(e_1, e_2)$.

\begin{thm}
Let $e_1<e_2$, then Case b2) holds true.  Let $e_1\geq e_2$, then: 

(i) If for some $i\in\ens{m}$ such that $p^X_i\wedge p^Y_i=e_1$,  one has $p^X_i\neq p^Y_i$, then Case a) holds true or so does its symmetric version: there exists $u\in L_U$ such that $\frac{u_1}{p^Y_1}+\dots+\frac{u_m}{p^Y_m} =1$ and $\frac{u_1}{p^X_1}+\dots+\frac{u_m}{p^X_m}<1$.

(ii) Otherwise, i.e., if for all $i\in\ens{m}$ such that $p^X_i\wedge p^Y_i=e_1$, one has 
$p^X_i=p^Y_i$, then if $e_1>e_2$  Case b1) holds true, while if $e_1=e_2$, then so 
does Case b2).

\begin{proof}

First, for any $0<\delta<1$, let $e_{\max,\delta}$, $e_{1,\delta}$, $e_{2,\delta}$ and $e_{\delta}(i,j)$ be defined just as $e_{\max},e_1,e_2$ and $e(i,j)$ but replacing $p^Y_i$ with $\delta p^Y_i$, for all $i\in\ens{m}$. 
Next, from the very definition of Case a): There exists $u\in L_U$ such that $\frac{u_1}{p^X_1}+\dots+\frac{u_m}{p^X_m} =1$ and $\frac{u_1}{p^Y_1}+\dots+\frac{u_m}{p^Y_m}<1$. Letting $\delta_0:=\frac{u_1}{p^Y_1}+\dots+\frac{u_m}{p^Y_m}$, we have $\frac{u_1}{\delta_0 p^Y_1}+\dots+\frac{u_m}{\delta_0 p^Y_m}=1$ so $e_{\max,\delta_0}\geq e_{\max}$ and therefore (clearly, $e_{\max, \delta}$ is  non-decreasing in $\delta$) $e_{\max,\delta_0}= e_{\max}$. So when Case a) occurs there exists $0<\delta_0<1$, such that 
for all $\delta\in (\delta_0,1], e_{\max,\delta}=e_{\max}$, and one can easily check the converse. 
A similar result continues to hold for the symmetric version of Case a).   

We can now prove the statement of the theorem by distinguishing the following four occurrences.

(1) Let $e_1<e_2$. Let $0<\delta_0<1$ be close enough to $1$ such that for any 
$\delta\in (\delta_0,1]$, the set of 
pairs $i,j\in\ens{m}$ such that $\substack{p^X_i<\,p^X_j\\\rotatebox{90}{$\scriptstyle >$}\quad\: \rotatebox{90}{$\scriptstyle <$}\\\: p^Y_i>\,p^Y_j}$ is equal to the set of $i,j\in\ens{m}$ such that 
$\substack{p^X_i<\,p^X_j\\\rotatebox{90}{$\scriptstyle >$}\quad\: \rotatebox{90}{$\scriptstyle <$}\\\: \delta p^Y_i>\,\delta p^Y_j}$.   Since for every $i,j$ in this set, it is immediate to check that $e(i,j)>e_{\delta}(i,j)$, the maximums satisfy $e_2>e_{\delta,2}$. Since $e_1<e_2$, by continuity, for $\delta$ close enough to $1$, $\max(e_{\delta,1}, e_{\delta,2})=e_{\delta,2}$ so $e_{\delta,\max}<e_{max}$, hence we are in Case b). There are $i,j\in \ens{m}$ such that $e_{\max}=e_2=e(i,j)$, so 
$i,j$ are in $I$, but $p^X_i<p^X_j$ so we are 
in Case b2).

(2) Let $e_1\geq e_2$, and let there exist $i\in\ens{m}$ such that $p^X_i\wedge p^Y_i=e_1$ and $p^X_i\neq p^Y_i$, say, $p^X_i<p^Y_i$.  Then, the very definition of Case a) is verified with the vector $u\in\mathbb{R}^m$ 
having coordinates  equal to zero except for $u_i=p^X_i$.  If instead, $p^X_i>p^Y_i$ then the symmetric 
case holds true.  

(3) Let $e_1>e_2$ and let for all $i\in\ens{m}$ such that $p^X_i\wedge p^Y_i=e_1$, $p^X_i=p^Y_i$. By continuity, for $\delta$ close enough to $1$, $\max(e_{\delta,1}, e_{\delta,2})=e_{\delta,1}=\delta e_{\max}$ so we are in Case b). Additionally, one verifies that under our assumptions $I$ is restricted to the set 
of $i\in\ens{m}$ such that $p^X_i=p^Y_i=e_{\max}$.  Therefore, we are, in fact,  in Case b1).

(4) Let $e_1=e_2$ and let for all $i\in\ens{m}$ such that $p^X_i\wedge p^Y_i=e_1$, $p^X_i=p^Y_i$. From what is done above, we see that for $\delta$ close enough to $1$, $e_{\delta,\max}<e_{\max}$ hence we are in Case b). Once again, since there are $i,j\in \ens{m}$ such that $e_{\max}=e_2=e(i,j)$, we are in Case b2).
\end{proof}

\end{thm}
To present another explicit example, let us fully corner the case $m=2$, 
with $p_1^X, p_2^X, p_1^Y,$ and $p_2^Y$.  The following completely describes the various cases:  
\begin{itemize}

\item If $p^X_1=p^Y_1$, then (since, necessarily, $p^X_2=p^Y_2$) $e_{max}=\max(p^X_1, p^X_2)=\max(p^X_1, 1-p^X_1)$ and we are in Case b1).

\item  If $p^X_1\neq p^Y_1$ and $1/2\in (\min(p^X_1, p^Y_1), \max(p^X_1,p^Y_1))$, then 
$$e_{max}=\max(\min(p^X_1, p^Y_1), \min(p^X_2, p^Y_2))=\max(\min(p^X_1, p^Y_1), \min(1-p^X_1, 1-p^Y_1)),$$ and we are in Case a) or its symmetric.

\item  If $p^X_1\neq p^Y_1$ and $1/2\notin (\min(p^X_1, p^Y_1), \max(p^X_1, p^Y_1))$, then 
$$e_{max}=\frac{p^X_1 p^Y_1(p^X_2-p^Y_2)+p^X_2 p^Y_2(p^Y_1-p^X_1)}{p^Y_1 p^X_2-p^X_1 p^Y_2}= p^X_1 p^Y_1 + p^X_2 p^Y_2 = p^X_1 p^Y_1 + (1-p^X_1)(1-p^Y_1), 
$$ 
and we are in Case b2).

\end{itemize}

\section{The limiting law}

It is clear, from the previous section, that the proper way to center (and normalize) $LCI_n$ is via \begin{align*}
Z_n & =\frac{LCI_n-n e_{\max}}{\sqrt{n}} \\
& =\max_{\substack {\lambda\in \Lambda^2}} \sum_{i=1}^m\left[\left(\sqrt{n} p^X_i\lambda^X_i+\widetilde{V}^{n,X}_i(\lambda^X)\right)\wedge \left(\sqrt{n} p^Y_i\lambda^Y_i+\widetilde{V}^{n,Y}_i(\lambda^Y)\right)\right]-\sqrt{n}e_{\max}.
\end{align*}

\noindent
Let also
\begin{align*}
Z^c_n & =\frac{LCI^c_n-n e_{\max}}{\sqrt{n}} \\
& =\max_{\substack {\lambda\in \Lambda^2}} \sum_{i=1}^m\left[\left(\sqrt{n} p^X_i\lambda^X_i+V^{n,X}_i(\lambda^X)\right)\wedge \left(\sqrt{n} p^Y_i\lambda^Y_i+V^{n,Y}_i(\lambda^Y)\right)\right]-\sqrt{n}e_{\max},
\end{align*}
from (\ref{closetoc}) we have 
\begin{equation} \label{probsection}
|Z_n-Z^c_n| \leq \frac{m}{\sqrt{n}},
\end{equation}and therefore the convergence in distribution of $Z^c_n$ will imply the convergence, in distribution, of $Z_n$ towards the same limit.

\subsection{Statement of the theorem}

Below is the main result of the paper.  In this statement, 
the covariance matrices of the Brownian motions stem from the covariance matrix of the 
rescaled variables $(\mathds{1}_{X_{k}=i})_{i \in I}$ 
(resp.~$\mathds{1}_{Y_{k}=i}, i\in I$) used to construct the polygonal 
approximations $B^{n,\bullet}_i$ (here, and throughout, 
$\bullet$ is short for either $X$ or $Y$).  Indeed, note that $\mathds{E}\left(\frac{(\mathds{1}_{X_{k}=i}-p^X_i)(\mathds{1}_{X_{k}=j}-p^X_j)}{\sqrt{p^X_i(1-p^X_i)}\sqrt{p^X_j(1-p^X_j)}}\right)=-\sqrt{\frac{p^X_i p^X_j}{(1-p^X_i)(1-p^X_j)}}$ (with a similar result for $Y$).

\begin{thm} \label{theoreme}
Let $B^X$ and $B^Y$ be two independent $\vert I\vert$-dimensional Brownian motions defined on 
$[0,1]$ with respective covariance matrix $C^X$ defined by $C^X_{i,i}=1$ and $C^X_{i,j}=-\sqrt{\frac{p^X_i p^X_j}{(1-p^X_i)(1-p^X_j)}}$, for $i\neq j$ in $I$, and $C^Y$ defined in a similar fashion, replacing $p^X_i$ by $p^Y_i$ and $p^X_j$ by $p^Y_j$. For all $\lambda\in K_{\Lambda^2}$ and $i\in I$, set \begin{align*}
V^X_i(\lambda^X)=\sqrt{p^X_i(1-p^X_i)}\left(B^X_i\left(\sum_{j=1}^i\lambda^X_j\right)-B^X_i\left(\sum_{j=1}^{i-1}\lambda^X_j\right)\right),\\
V^Y_i(\lambda^Y)=\sqrt{p^Y_i(1-p^Y_i)}\left(B^Y_i\left(\sum_{j=1}^i\lambda^Y_j\right)-B^Y_i\left(\sum_{j=1}^{i-1}\lambda^Y_j\right)\right).
\end{align*}
If there exists $u\in L_U$ such that $\frac{u_1}{p^X_1}+\dots+\frac{u_m}{p^X_m} =1$ and 
$\frac{u_1}{p^Y_1}+\dots+\frac{u_m}{p^Y_m}<1$ (Case a)), then  \begin{equation}
\frac{LCI_n-ne_{\max}}{\sqrt{n}}\xRightarrow[n \to \infty]{}Z^a:=\max_{\substack {\lambda^X\in J}} \sum_{i\in I} V^X_i(\lambda^X), 
\end{equation}
where $J$ is given by \eqref{defJ}. 

\noindent
If for all $u\in L_U$, $\frac{u_1}{p^X_1}+\dots+\frac{u_m}{p^X_m} = \frac{u_1}{p^Y_1}+\dots+\frac{u_m}{p^Y_m}=1$ (Case b)), then \begin{equation}
\frac{LCI_n-ne_{\max}}{\sqrt{n}}\xRightarrow[n \to \infty]{}Z^b:=\max_{\substack {\lambda\in K_{\Lambda^2}}} \mathfrak{m}\left(V^X(\lambda^X),V^Y(\lambda^Y)\right), 
\end{equation}
where $\mathfrak{m}$ is given by \eqref{defm}.  
\end{thm}

At this point, one can remark that $e_{\max}$ is invariant with respect to the order in which 
the letters are chosen, and that 
both in Case a) and Case b1), the above limiting laws are invariant as well (to see this fact in Case a), recall Lemma~\ref{icasa}). Therefore, in Case a) and Case b1), no matter the prescribed order (increasing, decreasing, etc..) the asymptotic behavior of the length of the corresponding optimal alignments is the same.  
We refer the reader to Section~\ref{genblock} for more general results of this flavor.  

In Case b2) it is less clear that the limiting distribution is permutation-invariant as it might not just boil down 
to $\mathfrak{m}(\nu)$.  Indeed, in Case b2) the limiting law can be written as 
the law of   
$$Z=\max_{\lambda\in K_{\Lambda^2}} \sum_{\substack{i\in\{1, \dots, m\}\\ \bullet \in\{X,Y\}}} V(\lambda)^\bullet_i,$$ 
where $V(\lambda)$ is in $\left(\mathbb{R}^m\right)^2$, 
and defined via 
$$V^\bullet(\lambda)_i=B^\bullet_i\left(\sum_{j=1}^{i}\lambda^\bullet_j\right)-B^\bullet_i\left(\sum_{j=1}^{i-1}\lambda^\bullet_j\right),$$ 
where the $B^\bullet_i$ are Brownian motions which are, up to a multiplicative factor, as in our main theorem.  Further introducing, for any permutation 
$\sigma$ of $\left\{1, \dots, n\right\}$, $V_\sigma(\lambda)$ defined via 
$$V^\bullet_\sigma(\lambda)_i=B^\bullet_i\left(\sum_{j=1}^{\sigma^{-1}(i)}\lambda^\bullet_{\sigma(j)}\right)-B^\bullet_i\left(\sum_{j=1}^{\sigma^{-1}(i)-1}\lambda^\bullet_{\sigma(j)}\right),$$ 
we have $V(\lambda)=V_{\text{Id}}(\lambda)$, where $\text{Id}$ is the identity permutation.  
When the letters are not required to be 
non-decreasing, but instead follow an 
order given by $\sigma$, the limiting law is simply the law of 
$Z_\sigma:=\max_{\lambda\in K_{\Lambda^2}} \sum_{\substack{i\in\{1, \dots, m\}\\ \bullet \in\{X,Y\}}} V_\sigma(\lambda)^\bullet_i$. 
It is still not that clear whether or not 
this last quantity depends on $\sigma$.  For example, if $m=3$ and $K_{\Lambda^2}=\Lambda^2$ and $B^X_1$ is a standard Brownian motion, while all others are null, define $\sigma$ by $\sigma(1)=2, \sigma(2)=1, \sigma(3)=3$, then with probability one $Z_{\sigma}>Z_{\text{Id}}$.  However, in Case b2), it is actually 
not possible to have $K_{\Lambda^2}=\Lambda^2$ (and also to have only one non null Brownian motion) but this shows that a general argument for the validity of the permutation-invariance is not that transparent.

\subsection{Proof of Theorem~\ref{theoreme}}

The proof of this theorem is based on a non-probabilistic lemma. First, let $E^\eta_n$ be the set of all continuous functions $b$ from $[0,1]$ into $\mathbb{R}$ such that: for all $x,y$ in $[0,1]$, $\vert b(y)-b(x)\vert \leq 
\left(n^\eta\sqrt{\vert y-x\vert}+n^{\eta-1/2}\right)/2$. Then, for all $b\in \left(E^\eta_n\right)^{m}$, 
$i\in\ens{m}$ and $\lambda\in\Lambda$, set $v^b_i(\lambda)=b_i(\lambda_1+\dots+\lambda_i)-b_i(\lambda_1+\dots+\lambda_{i-1})$, and for all $b^X, b^Y\in \left(E^\eta_n\right)^{m}$ and $\lambda\in\Lambda^2$ let 
\begin{equation*}
z_n(\lambda)=\sum_{i=1}^m\left[\left(\sqrt{n} p^X_i \lambda^X_i+v^{b^X}_i(\lambda^X)\right)
\wedge \left(\sqrt{n} p^Y_i \lambda^Y+v^{b^Y}_i(\lambda^Y)\right)\right]-\sqrt{n}e_{\max}.
\end{equation*}

One can think of $b_i^X$ (resp. $b_i^Y$) as $\sqrt{p_i^X(1-p_i^X)}B_i^{n,X}(\omega)$ 
(resp.~$\sqrt{p_i^Y(1-p_i^Y)}B^{n,Y}(\omega)$) for a fixed $\omega\in A^\eta_n$, where the symbol $b^X$ (resp.~$b^Y$) is used for ease of notation 
and in order to emphasize the non-probabilistic nature of the proof. For further ease of notation, 
we omit the dependency in $b^X$ and $b^Y$ in the notation $z_n$. 
This omission is also present in $v$ and $v^X$ is just short for $v^{b^X}$ (similarly 
with $Y$), and further write $v(\lambda):=\left(v^X(\lambda^X),v^Y(\lambda^Y)\right)$.
\noindent
In Case a), for all $\lambda^X \in \Lambda$, let 
\begin{equation}\label{defza}
z^{a}(\lambda^X):=\sum_{i\in I} v^{X}_i(\lambda^X).
\end{equation}
In Case b), for all $\lambda\in \Lambda^2$, let 
\begin{equation}\label{defzb}
z^{b}(\lambda)=\mathfrak{m}\left(v^X(\lambda^X),v^Y(\lambda^Y)\right).  
\end{equation}

Next, let us finally present two simple inequalities stemming from the very 
definition of $E^\eta_n$, often used in the sequel, 
which are valid for all $b\in E^\eta_n$, $\lambda,\lambda'\in \Lambda$, $i\in\ens{m}$, 
$\bullet\in\{X,Y\}$, namely,
\begin{equation}\label{boundv1}
\vert v^\bullet_i(\lambda^\bullet)\vert \leq \frac{n^\eta \sqrt{\lambda^\bullet_i}
+n^{\eta-1/2}}{2}\quad \text{and in particular} \quad \vert v^\bullet_i(\lambda^\bullet)\vert 
\leq n^\eta,
\end{equation}
\begin{equation}\label{boundv2}
 \vert v^\bullet_i(\lambda^\bullet)-v^\bullet_i(\lambda'^\bullet)\vert\leq n^\eta \sqrt{\max_{i\in\ens{m}}|\lambda_1+\dots+\lambda_i-\lambda'_1-\dots-\lambda'_i|}+n^{\eta-1/2}\leq 
n^\eta \sqrt{m \vert \|\lambda-\lambda'\|_\infty}+n^{\eta-1/2}.
\end{equation}

\begin{lem}\label{lemme}
There exists a sequence $(\varepsilon_n)_{n\geq 1}$ of positive reals converging to zero and 
such that for all $n\geq 1$ and $b^X, b^Y\in \left(E^\eta_n\right)^{m}$, 
either $\vert \max_{\lambda \in \Lambda^2} z_n(\lambda)-\max_{\lambda\in J} z^{a}(\lambda)\vert\leq \varepsilon_n$, or $\vert \max_{\lambda \in \Lambda^2} z_n(\lambda)-\max_{\lambda\in K_{\Lambda^2}} z^{b}(\lambda)\vert\leq \varepsilon_n$, in Case a) or b), respectively.
\end{lem}

The proof of this crucial lemma is delayed 
to the next subsections, and instead we turn our attention to the proof of the main theorem.

\begin{proof}[Proof of Theorem~\ref{theoreme}]
Let us assume that Case b) is occurring.
Let \begin{equation*}
Z^b_n=\max_{\substack {\lambda\in K_{\Lambda^2}}} 
\mathfrak{m}\left(V^{n,X}(\lambda^X),V^{n,Y}(\lambda^Y)\right).
\end{equation*}
For all $\omega\in A^{\eta}_n$, $B^{n,X}(\omega)$ and $B^{n,Y}(\omega)$ are in $E^{\eta}_n$ so by Lemma~\ref{lemme}, $\vert Z^c_n(\omega)-Z^b_n(\omega)\vert \leq \varepsilon_n$. So $\left\vert Z^c_n-Z^b_n\right\vert\mathds{1}_{A^\eta_n}\leq \varepsilon_n$, but $Z^c_n-Z^b_n=\left(Z^c_n-Z^b_n\right)\mathds{1}_{A^\eta_n}+\left(Z^c_n-Z^b_n\right)\mathds{1}_{(A^{\eta}_n)^c}$, where this second term tends to zero in probability, therefore so does $Z^c_n-Z^b_n$. Next, by Donsker's theorem and the continuity of $\mathfrak{m}$ (recalling Lemma~\ref{exprm}), $Z^b_n$ tends to $Z^b$ in distribution, so does $Z^c_n$ and finally so is the case for $Z_n$, recalling (\ref{probsection}). The proof  in the Case a) is analogous and therefore omitted.
\end{proof}

Let us now turn to the proof of Lemma~\ref{lemme}. The method of proof goes as follows: Maximizing $z_n(\lambda)$ is equivalent to maximizing
\begin{equation*}
z_n(\lambda)/\sqrt{n}=\sum_{i=1}^m\left[\left(p^X_i \lambda^X_i+v^{b^X}_i(\lambda^X)/\sqrt{n}\right)\wedge \left(p^Y_i \lambda^Y+v^{b^Y}_i(\lambda^Y)/\sqrt{n}\right)\right]-e_{\max}, 
\end{equation*}
which converges, as $n$ goes to infinity, to $f(\lambda)-e_{\max}$.  So one can expect that $\lambda$ must  "almost" be maximizing $f$, i.e., be in or "close to" the set $K_{\Lambda^2}$. In Case a), we bound the 
maximum by taking the maximum over two sets which are closer and closer to the set $J$.  
In  Case b), first write $\lambda=\lambda^{K_{\Lambda^2}}+\lambda^r$ 
(actually dealing with a $\lambda-a$ in order to have a vector space, but the idea is the same), then ignore the small  perturbation term $\lambda^r$ in $v$, and the idea is (roughly) to fix $\lambda^{K_{\Lambda^2}}$ and to find the maximum over $\lambda^r$.  In both cases, the end of the proof consists in showing how the maximum of the relevant function ($z^a$ or $z^b$) over a set of parameters that "tends to" a limiting set goes to the maximum over this limiting set.

\subsection{Proof of Lemma~\ref{lemme}, Case a)}
\subsubsection{Restriction to $I$}

First, fix $b=(b^X,b^Y)\in\left(\left(E^\eta_n\right)^{m}\right)^2$.  Next, 
for ease of notation, omit in the sub-index $b$ in $z$ and $v$.  
Roughly speaking, we begin by proving that any $\lambda$ maximizing $z_n$ must have "small" coordinates outside of $I$, and therefore we can "replace" the variations $v^._i$, for $i\notin I$, 
by zero.

Let \begin{equation}
p^X_{\text{sec}}= \begin{cases} 
      \max_{i\notin I} {p^X_i} & I\neq \ens{m}, \\
      0 & I=\ens{m} \\
   \end{cases}.
\end{equation}
Let us assume first that $I\neq\ens{m}$. Then by Lemma~\ref{icasa}, $p^X_{\text{sec}}<p^X_{\max}$. Our first observation is that if $\lambda$ maximizes $z_n$, i.e., if 
$z_n(\lambda)=\max_{\lambda\in \Lambda^2}z_n(\lambda)$, then 
\begin{equation}\label{boundofs}
s:=\sum_{i\notin I} \lambda^X_i\leq \frac{2m n^{\eta-1/2}}{p^X_{\max}-p^X_{\text{sec}}}.
\end{equation}

In words, the above indicates that the contribution of the letters not in $I$ is, as expected, 
very limited. 
To prove this inequality, note that on the one hand (recalling Lemma~\ref{icasa} and \eqref{boundv1}),
\begin{equation*}
z_n(\lambda)\leq \sum_{i=1}^m\left(\sqrt{n} p^X_i \lambda^X_i+v^{b}_i(\lambda)\right)-\sqrt{n}p^X_{\max}\leq \sqrt{n}\left(p^X_{\max}(1-s)+p^X_{\text{sec}}s\right)+m n^\eta-\sqrt{n}p^X_{\max},
\end{equation*}
while on the other hand, for $\tilde{\lambda}\in K_{\Lambda^2}$, using \eqref{boundv1} 
and the elementary inequality \eqref{elem},
\begin{equation}
\label{eqK}z_n(\lambda)\geq z_n(\tilde{\lambda})\geq \sqrt{n} f(\tilde{\lambda})-m n^\eta 
-\sqrt{n}p^X_{\max}=-m n^\eta. 
\end{equation}
The inequality \eqref{boundofs} follows, and it therefore allows, 
for $i\notin I$, to replace 
the terms $v_i^X(\lambda^X)$ 
by zero.  More precisely, let for all $\lambda\in\Lambda^2$, 
\begin{equation*}
z^I_n(\lambda)=\sum_{i\in I}\left[\left(\sqrt{n} p^X_i \lambda^X_i+v^X_i(\lambda^X)\right)\wedge \left(\sqrt{n} p^Y_i \lambda^Y_i+v^Y_i(\lambda^Y)\right)\right]+\sum_{i\notin I}\left[\left(\sqrt{n} p^X_i \lambda^X_i\right)\wedge \left(\sqrt{n} p^Y_i \lambda^Y_i+v^Y_i(\lambda^Y)\right)\right]-\sqrt{n}e_{\max},
\end{equation*}
then as shown next, 
\begin{equation}
\label{ineqI}\left\vert \max_{\lambda\in \Lambda^2}z_n(\lambda)-\max_{\lambda\in \Lambda^2}z^I_n(\lambda)\right\vert \leq \frac{\vert I^c\vert}{2}\left(n^\eta\sqrt{\frac{2m n^{\eta-1/2}}{p^X_{\max}-p^X_{\text{sec}}}}+n^{\eta-1/2}\right),
\end{equation}
and this inequality remains true when $I=\ens{m}$ (since then $\max_{\lambda\in \Lambda^2}z_n(\lambda)
=\max_{\lambda\in \Lambda^2}z^I_n(\lambda)$ and $\vert I^c\vert=0$).

Indeed, let $\lambda\in\Lambda^2$ be such that $z_n(\lambda)=\max_{\lambda\in \Lambda^2}z_n(\lambda)$. Using \eqref{elem} along with \eqref{boundv1} ($\lambda^X_i\leq {2m n^{\eta-1/2}}/(p^X_{\max}-p^X_{\text{sec}})$, for all $i\notin I$), it follows that 
$$\max_{\lambda\in \Lambda^2}z^I_n(\lambda)\geq z^I_n(\lambda)\geq \max_{\lambda\in \Lambda^2}z_n(\lambda)-\frac{\vert I^c\vert}{2}\left(n^\eta\sqrt{\frac{2m n^{\eta-1/2}}{p^X_{\max}-p^X_{\text{sec}}}}+n^{\eta-1/2}\right).$$ 
Moreover, let $\tilde{\lambda}\in\Lambda^2$ be such that $\max_{\lambda\in \Lambda^2}z^I_n(\lambda)=z^I_n(\tilde{\lambda})$. Then, just as in proving \eqref{boundofs}, it follows that 
$\sum_{i\notin I} \tilde{\lambda}^X_i\leq {2 \vert I\vert n^{\eta-1/2}}/(p^X_{\max}-p^X_{\text{sec}})$. Hence 
$$\max_{\lambda\in \Lambda^2}z_n(\lambda)\geq z_n(\tilde{\lambda})\geq \max_{\lambda\in \Lambda^2}z^I_n(\lambda)-\frac{\vert I^c\vert}{2}\left(n^\eta\sqrt{\frac{2m n^{\eta-1/2}}{p^X_{\max}-p^X_{\text{sec}}}}+n^{\eta-1/2}\right),$$ 
which completes the proof.

\subsubsection{Bounds on the maximum with different sets of constraints}
Let us next define two sets "close" to $J$. To do so, let $S_n=2 \vert I\vert ^2 n^{\eta-1/2}$, let $C_I=\sum_{i\in I} \frac{1}{p^Y_i}$, let $T_n= C_I 2 n^{\eta-1/2}$, and finally let 
\begin{equation*}
J_n^{+}=\left\{\lambda^X\in \Lambda\::\: \sum_{i\in I} \frac{\lambda^X_i}{p^Y_i}\leq \frac{1+S_n}{p^X_{\max}}\right\},
\end{equation*}
and 
\begin{equation*}
J_n^{-}=\left\{\lambda^X\in \Lambda\::\: \sum_{i\in I} \frac{\lambda^X_i}{p^Y_i}\leq 
\frac{1-T_n}{p^X_{\max}}\right\}.
\end{equation*}

Note that by Lemma~\ref{icasa}, setting $\delta_{i_1}=\left(\mathds{1}_{i=i_1}\right)_{i\in\ens{m}}$, 
$\delta_{i_1}\in J_n^{-}$ eventually.
We show, in this part of the proof, that 
\begin{equation}\label{gendarmes}
\max_{\lambda\in J_n^{-}}z^a(\lambda)\leq \max_{\lambda\in \Lambda^2}z^I_n(\lambda)
\leq \max_{\lambda\in J_n^{+}}z^a(\lambda).
\end{equation}

Let us prove the upper bound first. Let $\lambda\in\Lambda^2$ be such 
that $z^I_n(\lambda)=\max_{\lambda\in \Lambda^2}z^I_n(\lambda)$, and 
let $S$ be the unique real such that 
\begin{equation*}
\sum_{i\in I} \frac{\lambda^X_i}{p^Y_i}=\frac{1+S}{p^X_{\max}}.
\end{equation*}

Then, there exists $i_0\in I$ such that, 
\begin{equation*}
\lambda^Y_{i_0} p^Y_{i_0}\leq \lambda^X_{i_0}p^X_{\max}-\frac{S}{\vert I\vert},
\end{equation*} 
since otherwise, $\sum_{i\in I} \lambda^Y_i>1$, which is a contradiction. 
Then, using the following inequalities, 
\begin{align*}
\forall i\in I\setminus\{i_0\}\quad \left(\sqrt{n} p^X_i \lambda^X_i+v^X_i(\lambda^X)\right)\wedge \left(\sqrt{n} p^Y_i \lambda^Y_i+v^Y_i(\lambda^Y)\right)& \leq \left(\sqrt{n} p^X_i \lambda^X_i+v^X_i(\lambda^X)\right),\\
\left(\sqrt{n} p^X_{i_0} \lambda^X_{i_0}+v^X_{i_0}(\lambda^X)\right)\wedge \left(\sqrt{n} p^Y_{i_0} \lambda^Y_{i_0}+v^Y_{i_0}(\lambda^Y)\right) & \leq \left(\sqrt{n} \left(\lambda^X_{i_0}p^X_{\max}-\frac{S}{\vert I\vert}\right)+v^Y_{i_0}(\lambda^Y)\right),\\
\forall i\notin I\quad\left(\sqrt{n} p^X_i \lambda^X_i\right)\wedge \left(\sqrt{n} p^Y_i \lambda^Y_i+v^Y_i(\lambda^Y)\right) & \leq \sqrt{n} p^X_i \lambda^X_i,
\end{align*} 
\noindent
leads to 
\begin{align*}
z^I_n(\lambda)&\leq \sqrt{n}\sum_{i=1}^m p^X_i \lambda^X_i+\sum_{i\in I\setminus\{i_0\}}\left(v^X_i(\lambda^X) +v^Y_{i_0}(\lambda^Y)\right)-\sqrt{n}\frac{S}{\vert I\vert}-\sqrt{n}e_{\max}\\
&\leq \sum_{i\in I\setminus\{i_0\}}\left(v^X_i(\lambda^X) +v^Y_{i_0}(\lambda^Y)\right)
-\sqrt{n}\frac{S}{\vert I\vert}\\
&\leq \vert I\vert n^\eta -\sqrt{n}\frac{S}{\vert I\vert}.
\end{align*} 
\noindent
Just as in obtaining the inequality \eqref{eqK}, we have $-\vert I\vert n^\eta\leq z^I_n(\lambda)$, hence 
$S\leq 2\vert I\vert ^2 n^{\eta-1/2}$, i.e., $\lambda^X\in J_n^{+}$, leading to conclude with the 
upper estimate:  
\begin{equation*}
\max_{\lambda\in \Lambda^2}z^I_n(\lambda)=z^I_n(\lambda)\leq \sqrt{n}f(\lambda^X)+z^a(\lambda^X)-\sqrt{n}e_{\max}\leq z^a(\lambda^X)\leq \max_{\lambda\in J_n^{+}}z^a(\lambda).
\end{equation*}

Let us now turn our attention to the lower bound. Let $\lambda^X\in J_n^{-}$ be such that $z^a(\lambda^X)
=\max_{\lambda\in J_n^{-}}z^a(\lambda)$. 
Since $$\sum_{i\in I}\left(p^X_{\max}\lambda^X_i+2 n^{\eta-1/2}\right)/p^Y_i\leq 1,$$ there 
exists $\lambda^Y\in \Lambda$ such that for $i\in I$, $\lambda^Y_i\geq \left(p^X_{\max}\lambda^X_i+2 n^{\eta-1/2}\right)/p^Y_i$ and for $i\notin I$, $\lambda^Y_i=0$. For all $i\in I$, 
\begin{equation*}
\sqrt{n} p^Y_i \lambda^Y_i+v^Y_i(\lambda^Y)\geq \sqrt{n} p^X_{\max}\lambda^X_i + 2n^{\eta}+ v^Y_i(\lambda^Y)\geq \sqrt{n} p^X_{\max} \lambda^X_i+v^X_i(\lambda^X)=\sqrt{n} p^X_i \lambda^X_i+v^X_i(\lambda^X).  
\end{equation*} 
Therefore, 

\begin{equation*}
z^I_n(\lambda)=\sum_{i\in I} \left(\sqrt{n} p^X_i \lambda^X_i+v^X_i(\lambda^X)\right) + 
\sum_{i\notin I} \left[(\sqrt{n} p^X_i \lambda^X_i)\wedge 0\right] -\sqrt{n} p^X_{\max}= \sum_{i\in I} v^X_i(\lambda^X)= z^a(\lambda^X)=\max_{\lambda\in J_n^{-}}z^a(\lambda),
\end{equation*}
and $\max_{\lambda\in J_n^{-}}z^a(\lambda)\le \max_{\lambda\in \Lambda^2}z^I_n(\lambda)$.

\subsubsection{End of the proof}

Both quantities $\vert \max_{\lambda\in J_n^{-}}z^a(\lambda)-\max_{\lambda\in J}z^a(\lambda)\vert$ 
and $\vert \max_{\lambda\in J_n^{+}}z^a(\lambda)-\max_{\lambda\in J}z^a(\lambda)\vert$ still need to be investigated. Let $C_1=\left(1-\frac{p^X_{\max}}{p^{Y}_{i_1}}\right)>0$. For $\lambda^X\in \Lambda$ and $t\in(0,1)$, let $\lambda^{X,t}=t \delta_{i_1} +(1-t)\lambda^X$. It is straightforward to prove that for all $n$ greater than some constant, depending only on $\eta$, $p^X$ and $p^Y$, and for all $\lambda^X\in J$, $\lambda^{X,\frac{T_n}{C_1}}$ is well defined, and is in $J_n^{-}$, while for 
all $\lambda^X\in J_n^{+}$, $\lambda^{X,\frac{2 S_n}{C_1}}\in J$.

\noindent
This is useful since for all $i\in \ens{m}$,
\begin{equation*}
\vert\lambda^X_1+\dots+\lambda^X_i-\lambda^{X,t}_1-\dots-\lambda^{X,t}_i\vert \leq 2 t,
\end{equation*}
and therefore, using \eqref{elem} along with \eqref{boundv2},
\begin{align*}
\max_{\lambda\in J}z^a(\lambda)-\max_{\lambda\in J_n^{-}}z^a(\lambda)\leq \vert I\vert \left(n^\eta\sqrt{\frac{2 T_n}{C_1}}+ n^{\eta-1/2}\right),\\
\max_{\lambda\in J_n^{+}}z^a(\lambda)-\max_{\lambda\in J}z^a(\lambda)\leq \vert I\vert \left(n^\eta\sqrt{\frac{4 S_n}{C_1}}+ n^{\eta-1/2}\right).
\end{align*}
Putting these two inequalities, together with \eqref{gendarmes}, leads to

\begin{equation*}
\left\vert \max_{\lambda\in \Lambda^2}z^I_n(\lambda) - \max_{\lambda\in J}z^a(\lambda)\right\vert 
\leq C_2 n^{\frac{6\eta-1}{4}}+\vert I\vert n^{\eta-1/2},
\end{equation*}
for some constant $C_2$ depending only on the $p$'s but need not be made explicit. 
The lemma is thus proved in this case.

\subsection{Proof of Lemma~\ref{lemme}, Case b)}
\subsubsection{Preliminaries}

Fix $b=(b^X,b^Y)\in\left(\left(E^\eta_n\right)^{m}\right)^2$. Just as in Case a), we omit in the 
notation 
the sub-index $b$. Let $E=\{x\in\mathbb{R}^{m}\::\: x_1+\dots+x_m=0\}$, let $K$ be the 
subspace of $E^2$ defined by
\begin{equation*}\label{defK}
K=\left\{x\in E^2\::\: \forall i \in I, p^X_i x^X_i=p^Y_i x^Y_i, \forall i\notin I, x^X_i=y^Y_i=0\right\},
\end{equation*}
and let $P$ (recalling the definition of $a$ following \eqref{defui}: $a\in K_{\Lambda^2}$, for all $i\in I, p^X_i a^X_i=p^Y_i a^Y_i>0$, for $i\notin I, a^\bullet_i=0$, and $f(a)=e_{\max}$) be given by:  
\begin{equation}\label{defP}
P=\left\{x\in E^2\::\: \forall i\in \ens{m}, x^X_i\geq -a^X_i, x^Y_i\geq -a^Y_i \right\}.
\end{equation}
Note that $\Lambda^2=a+P$. By definition of the case b), for all $\lambda\in K_{\Lambda^2}$, 
for all $i\in I$ $\lambda^X_i p^X_i=\lambda^Y_i p^Y_i$, while for all $i\notin I$, $\lambda^X_i=\lambda^Y_i=0$. Reciprocally, let $\lambda\in\Lambda^2$ such that for all $i\in I$ $\lambda^X_i p^X_i=\lambda^Y_i p^Y_i$ and for all $i\notin I$, $\lambda^X_i=\lambda^Y_i=0$, we show that $\lambda\in K_{\Lambda^2}$. Let $u\in\mathbb{R}^I$ be defined by $u_i=p^X_i\lambda^X_i-p^X_i a^X_i$ 
for all $i\in I$. We have that $u\cdot P^X=u\cdot P^Y=1-1=0$ so by 
Lemma~\ref{indep}, $u\cdot (1)_{i\in I}=0$, hence the result. 
This characterization of $K_{\Lambda^2}$, combined with $\Lambda^2=a+P$, 
gives us \begin{equation}\label{Klambda}
K_{\Lambda^2}=a+K\cap P.
\end{equation}

Since $p^X_i a^X_i=p^Y_i a^Y_i$, for all $i\in\ens{m}$, \begin{equation*}
z_n(a+x)=\sum_{i=1}^m\left[\left(\sqrt{n} p^X_i x^X_i+v^X_i(a^X+x^X)\right)\wedge \left(\sqrt{n} p^Y_i x^Y_i+v^Y_i(a^Y+x^Y)\right)\right].
\end{equation*}
Clearly, 
\begin{equation*}
\max_{\lambda\in \Lambda^2}z_n(\lambda)=\max_{\substack {x\in P}} z_n(a+x).
\end{equation*}
\noindent
Note also that for all $x\in \left(\mathbb{R}^m\right)^2$, $f(a+x)=f(a)+f(x)$ so by \eqref{Klambda} 
\begin{equation}\label{propkp}
\forall x \in P,\,f(x)\leq 0 \quad\text{and}\quad \left(f(x)=0\right) \iff \left(x\in K\cap P\right).
\end{equation} 
Our next result is an elementary projection result.
\begin{lem}\label{proj} 
There exists $C>0$ depending only on $p^X$ and $p^Y$ such that for all $x \in P$, 
there exist $x^{K\cap P}\in K\cap P$ and $x^r\in E^2$ 
such that $x=x^{K\cap P}+x^r$ and $\|x^r\|_{\infty}\leq -C f(x)$.
\end{lem}

\begin{proof}
Let $K^{\bot}$ be the orthogonal complement of $K$ in $E^2$ (for the usual Euclidean inner product defined on $E^2$ by, for $x,y\in E^2$, $x\cdot y:=x^X_1 y^X_1+\dots+x^X_m y^X_m+x^Y_1 y^Y_1+\dots+x^Y_m y^Y_m$). Let $x\in P$ (so $x\in E^2$) and let $(x^K,x^{K^{\bot}})$ be its orthogonal decomposition, i.e., $x^K\in K$, $x^{K^{\bot}}\in K^{\bot}$ and $x=x^K+x^{K^{\bot}}$. Without loss of generality, assume $x^{K^{\bot}}\neq 0$. For ease of notation, set $g=-f$. Let \begin{equation*}
a_{\min}=\min_{i\in I} a_i.
\end{equation*}
In order to bound the image of $x^{K^{\bot}}\!\!\!,$ we first rescale it to make it an 
element of $P$: it is easy to check that $y:=\left(\frac{a_{\min}}{\|x^{K^{\bot}}\|_{\infty}}\right)x^{K^{\bot}}\in P$. Now, consider the sphere, \begin{equation*}
S_{a_{\min}}:=\left\{z\in K^{\bot}\::\: \|z\|_{\infty} = a_{\min}\right\}.
\end{equation*}
Then, $S_{a_{\min}}\cap P$ is a non-empty compact set, so let \begin{equation*}
M=\min_{z\in S_{a_{\min}}\cap P} g(z).
\end{equation*}
Recalling \eqref{propkp}, $M>0$. Since $y\in S_{a_{\min}}\cap P$, $M\leq g(y)$ so that, using 
$g\left(x^{K^{\bot}}\right)=g(x)$,\begin{equation*}
\|x^{K^{\bot}}\|_{\infty}\leq \frac{a_{\min}}{M}g(x).
\end{equation*}
This is almost the desired result, except that $x^K$ might not be in $P$. Let us assume, firstly, that 
$g(x)\leq M$ (and therefore that $\|x^{K^{\bot}}\|_{\infty}\leq a_{\min}$). Let $x^{K\cap P}=\left(1-\frac{\|x^{K^{\bot}}\|_{\infty}}{a_{\min}}\right)x^K$ and let $x^r=\frac{\|x^{K^{\bot}}\|_{\infty}}{a_{\min}}x^K+x^{K^{\bot}}$. We next prove that $x^{K\cap P}\in K\cap P$. Since $x\in P$, for $i\in I$, \begin{align*}
\left(1-\frac{\|x^{K^{\bot}}\|_{\infty}}{a_{\min}}\right)x^K_i+\left(1-\frac{\|x^{K^{\bot}}\|_{\infty}}{a_{\min}}\right)x^{K^{\bot}}_i&\geq -\left(1-\frac{\|x^{K^{\bot}}\|_{\infty}}{a_{\min}}\right)a_i\\
x^{K\cap P}_i&\geq -a_i+ \frac{\|x^{K^{\bot}}\|_{\infty}}{a_{\min}}a_i - \left(1-\frac{\|x^{K^{\bot}}\|_{\infty}}{a_{\min}}\right)x^{K^{\bot}}_i\\
&\geq -a_i+ \|x^{K^{\bot}}\|_{\infty} - \left(1-\frac{\|x^{K^{\bot}}\|_{\infty}}{a_{\min}}\right)\|x^{K^{\bot}}\|_{\infty}\\
&\geq -a_i,
\end{align*}
and for $i\notin I$, $x^{K\cap P}_i=0$, since $x^{K\cap P}\in K$. So $x^{K\cap P}\in K\cap P$.

Let us turn to $x^r$. Since $a+x\in \Lambda^2$, $\|x\|_{\infty}\leq 1$. Moreover, $x^K$ is the orthogonal projection 
of $x$ so $\|x^K\|_{\infty}\leq \sqrt{2m}\|x\|_{\infty}\leq \sqrt{2m}$ and \begin{align*}
\|x^r\|_{\infty}&\leq \left(\frac{\sqrt{2m}}{a_{\min}}+1\right)\|x^{K^{\bot}}\|_{\infty}\\
&\leq \left(\frac{\sqrt{2m}}{a_{\min}}+1\right)\frac{a_{\min}}{M}g(x).
\end{align*}

Setting $C:=\left(\sqrt{2m}+a_{\min}\right)/M$, we have just proved that if $g(x)\leq M$, then 
there exist suitable $x^{K\cap P}$ and $x^r$ satisfying the lemma.  
Finally, if $g(x)>M$, we let $x^{K\cap P}=0$ and 
$x^r=x$, so that $\|x^r\|_{\infty}\leq 1<g(x)/M<Cg(x)$ which completes the proof. 
\end{proof}

\subsubsection{Separation of the parameters}
To begin with, we prove that $\max_{x\in P}z_n(a+x)$ can be written as a maximum over two 
kind of parameters, one belonging to $K$ in the variations $v^._i$, the other one being a 
small remaining term.

Let $x\in P$ be such that $z_n(a+x)=\max_{\lambda\in \Lambda^2}z_n(\lambda)$. Then, 
\begin{equation*}
-mn^\eta \leq z_n(a)\leq z_n(a+x)\leq \sqrt{n}f(x)+mn^{\eta},
\end{equation*}
and so 
\begin{equation}
-f(x)\leq 2mn^{\eta-1/2}.\numberthis \label{boundf}
\end{equation}
Now, let
\begin{equation*}
D=\left\{(x^{K\cap P},x^r)\in (K\cap P)\times E^2 \::\: x^{K\cap P}+x^r\in P\right\},
\end{equation*}
and, recalling the constant $C$ from Lemma \ref{proj}, let
\begin{equation*}
D_n=\left\{(x^{K\cap P},x^r)\in (K\cap P)\times E^2 \::\: \|x^r\|_{\infty} \leq 2Cmn^{\eta-1/2}, 
x^{K\cap P}+x^r\in P\right\}.
\end{equation*}
Then, for all $(x^{K\cap P},x^r)\in D$, set
\begin{align*}
\overline{z}_n(x^{K\cap P},x^r)&=z_n(a+x^{K\cap P}+x^r)\\ &=\sum_{i=1}^m\Bigg[\left(\sqrt{n} p^X_i x^{r,X}_i+v^X_i(a^X+x^{K\cap P,X}+x^{r,X})\right)\wedge \left(\sqrt{n} p^Y_i x^{r,Y}_i+v^Y_i(a^Y+x^{K\cap P,Y}+x^{r,Y})\right)\Bigg],
\end{align*} and applying Lemma~\ref{proj} to \eqref{boundf} gives 
$\max_{x\in D_n}\overline{z}_{n}(x)=\max_{x\in P}z_n(a+x)$.

Let us next define a slight modification of $\overline{z}_{n}$ by letting, for 
all $(x^{K\cap P},x^r)\in D_n$,
\begin{equation*}
\overline{z}'_n(x^{K\cap P},x^r) =
\sum_{i=1}^m\Bigg[\left(\sqrt{n} p^X_i x^{r,X}_i+v^X_i(a^X+x^{K\cap P,X})\right)\wedge \left(\sqrt{n} p^Y_i x^{r,Y}_i+v^Y_i(a^Y+x^{K\cap P,Y})\right)\Bigg].
\end{equation*} The parameters are now "separated".
For all $(x^{K\cap P},x^r)\in D_n$, by \eqref{boundv2},\begin{equation*}
\left\vert \overline{z}'_n(x^{K\cap P},x^r)-\overline{z}_n(x^{K\cap P},x^r)\right\vert
\leq m\left(n^{\eta}\sqrt{2Cm^2n^{\eta-1/2}}+n^{\eta-1/2}\right),
\end{equation*}
so that
\begin{equation}\label{approxd}
\left\vert \max_{x\in P}z_{n}(a+x)-\max_{x\in D_n}\overline{z}'_{n}(x)\right\vert=\left\vert 
\max_{x\in D_n}\overline{z}_{n}(x)-\max_{x\in D_n}\overline{z}'_{n}(x)\right\vert \leq m\left(n^{\eta}\sqrt{2Cm^2n^{\eta-1/2}}+n^{\eta-1/2}\right).
\end{equation}

\subsubsection{Independence of the parameters}

A major issue with $D_n$ is the condition $x^{K\cap P}+x^r\in P$. We would rather have a 
set of possible values for $x^r$ independent of the value of $x^{K\cap P}$. To try to achieve that 
goal, let 
\begin{equation*}
P_n=\left\{x\in E^2\::\: \forall i\in I, x^X_i\geq -a^X_i+2Cmn^{\eta-1/2}, 
x^Y_i\geq -a^Y_i+2Cmn^{\eta-1/2}, \forall i\notin I, x^X_i\geq 0, x^Y_i\geq 0 \right\}\subset P, 
\end{equation*}
and let $D'_n\subset D_n$ be given by
\begin{equation*}
D'_n=\left\{(x^{K\cap P_n},x^r)\in (K\cap P_n)\times E^2 \::\: \|x^r\|_{\infty} 
\leq 2Cmn^{\eta-1/2}, x^{K\cap P_n}+x^r\in P\right\}.
\end{equation*}
Now, recalling the definition $E'=\left\{x\in E : \forall i\in I^c, x_i\geq 0 \right\}\subset E$, 
we have that
\begin{equation*}
D'_n=\left\{(x^{K\cap P_n},x^r)\in (K\cap P_n)\times E'^2 \::\: \|x^r\|_{\infty} \leq 2Cmn^{\eta-1/2}\right\}.
\end{equation*}
For $(x^{K\cap P},x^r)\in D_n$, and for $n$ large enough so that $\frac{2Cmn^{\eta-1/2}}{a_{\min}}\leq 1$, it follows that, letting $x'^{K\cap P}:=\left(1-\frac{2Cmn^{\eta-1/2}}{a_{\min}}\right)x^{K\cap P}$, $(x'^{K\cap P},x^r)\in D'_n$, so by \eqref{boundv2}
\begin{equation}\label{approxef}
\left\vert \max_{x \in D'_n}\overline{z}'_{n}(x)-\max_{x \in D_n}\overline{z}'_{n}(x)\right\vert 
\leq \vert I\vert \left(n^{\eta}\sqrt{\frac{2Cm^2 n^{\eta-1/2}}{a_{\min}}} +n^{\eta-1/2}\right).
\end{equation}

\subsubsection{Connections with the functions of Lemma~\ref{lemme}}

Let us now prove that for $n$ large enough, $\max_{x \in D'_n}\overline{z}^{\eta}_{n}(x)=\max_{\lambda \in a+K\cap P_n}\mathfrak{m}\left(v^X(\lambda^X),v^Y(\lambda^Y)\right)$.  
Fix $x^{K\cap P_n}\in K\cap P_n$. Applying the previous lemma to 
$\nu:=v(a+x^{K\cap P_n})$, since $\|\nu\|_{\infty}\leq n^{\eta}$, by Lemma~\ref{exprm}
\begin{equation*}
\max_{\substack{x^r\in E'^2 \\ \|x^r\|_{\infty}\leq 2Cmn^{\eta}}} \sum_{i=1}^m \left[\left(p^X_i x^{r,X}_i+\nu^X_i\right)\wedge \left(p^Y_i x^{r,Y}_i+\nu^Y_i\right)\right]=\max_{x^r\in E'^2} \sum_{i=1}^m \left[\left(p^X_i x^{r,X}_i+\nu^X_i\right)\wedge \left(p^Y_i x^{r,Y}_i+\nu^Y_i\right)\right]=\mathfrak{m}(\nu),
\end{equation*}
and so
 \begin{align*}
\max_{\substack{x^r\in E'^2 \\ \|x^r\|_{\infty}\leq 2Cmn^{\eta-1/2}}} \overline{z}'_n (x^{K\cap P_n},x^r)&=\max_{\substack{x^r\in E'^2 \\ \|x^r\|_{\infty}\leq 2Cmn^{\eta-1/2}}}  \sum_{i=1}^m \left[\left(\sqrt{n}p^X_i x^{r,X}_i+\nu^X_i\right)\wedge \left(\sqrt{n}p^Y_i x^{r,Y}_i+\nu^Y_i\right)\right]\\
&=\max_{\substack{x^r\in E'^2 \\ \|x^r\|_{\infty}\leq 2Cmn^{\eta}}} \sum_{i=1}^m \left[\left(p^X_i x^{r,X}_i+\nu^X_i\right)\wedge \left(p^Y_i x^{r,Y}_i+\nu^Y_i\right)\right]\\
&=\mathfrak{m}(\nu).
\end{align*}
Finally,
\begin{equation}\label{eqofmax}
\max_{x \in D'_n}\overline{z}'_{n}(x)=\max_{x\in K\cap P_n} \max_{\substack{x\in E'^2 \\ \|x\|_{\infty}\leq 2Cmn^{\eta-1/2}}}  \overline{z}'_n (x^{K\cap P_n},x^r)=\max_{\lambda \in a+K\cap P_n}\mathfrak{m}\left(v^X(\lambda^X),v^Y(\lambda^Y)\right).
\end{equation}

\subsubsection{End of the proof}

Just as done with \eqref{approxef}, 
\begin{equation*}
\left\vert \max_{\lambda \in a+K\cap P}\mathfrak{m}(v(\lambda))-\max_{\lambda \in a+K\cap P_n}\mathfrak{m}(v(\lambda))\right\vert \leq \vert I\vert \left(n^{\eta}\sqrt{\frac{2Cm^2 n^{\eta-1/2}}{a_{\min}}} +n^{\eta-1/2}\right),
\end{equation*}
and so, using 
\eqref{approxd}, \eqref{approxef} and \eqref{eqofmax} (recall that $a+K\cap P=K_{\Lambda^2}$),
\begin{equation*}
\left\vert \max_{x\in P}z_{n}(a+x)-\max_{\lambda \in K_{\Lambda^2}}\mathfrak{m}(v(\lambda))\right\vert \leq \left(\frac{2\vert I\vert}{\sqrt{a_{\min}}}+m\right)\sqrt{2Cm^2}n^{\frac{6\eta-1}{4}}+(2\vert I\vert+m)n^{\eta-1/2}.
\end{equation*}

\section{Consistency with previous results and generalizations}

\subsection{Two words with identical distributions}

As stated in the introductory section, Theorem~\ref{unfthm} and the conjectured Theorem~\ref{samedistthm} are 
consequences of our main theorem.  Indeed, let $X_k$ and $Y_k$ ($k=1,2,\dots$) have the same distribution, then note that \begin{equation*}
I=\left\{i\in\ens{m}\::\: p^X_i=p_{\max}\right\},
\end{equation*}and so the multiplicity $k^*$ of $p_{\max}$ is equal to $\vert I\vert$ and we are 
in Case b1).  
It is also clear that \begin{equation*}
K_{\Lambda^2}=\left\{\lambda\in \Lambda^2 \::\: \forall i\notin I, \lambda^X_i=
\lambda^Y_i=0\right\}^2.
\end{equation*}
In this case, Lemma~\ref{exprm} simplifies and gives $\mathfrak{m}(\nu)=S^X\wedge S^Y$, so our theorem states that the limiting distribution of $Z_n$ is
\begin{align*}\label{distsamedist}
&\max_{\substack {\lambda\in K_{\Lambda^2}}} \sqrt{p_{\max}(1-p_{\max})}\left[\left(\sum_{i\in I} B^X_i\left(\sum_{j=1}^i\lambda^X_j\right)-B^X_i\left(\sum_{j=1}^{i-1}\lambda^X_j\right)\right)\wedge \left(\sum_{i\in I} B^Y_i\left(\sum_{j=1}^i\lambda^Y_j\right)-B^Y_i\left(\sum_{j=1}^{i-1}\lambda^Y_j\right)\right)\right] \\
=& \max_{0=t_0\leq t_1\leq\dots\leq t_{k^*}=1} \sqrt{p_{\max}(1-p_{\max})}\left[\left(\sum_{i=1}^{k^*} \left(B^X_i(t_i)-B^X_i(t_{i-1})\right)\right)\wedge \left(\sum_{i=1}^{k^*} \left(B^Y_i(t_i)-B^Y_i(t_{i-1})\right)\right)\right],\numberthis
\end{align*}
where $B^X$ and $B^Y$ are two independent $k^*$-dimensional Brownian motions on $[0,1]$ with respective 
covariance matrix defined in Theorem~\ref{theoreme}.  
The proof of Corollary 3.3 in \cite{HL1} shows that, by writing $B^X$ and $B^Y$ as 
linear combinations of independent standard Brownian motions, \eqref{distsamedist} is identical 
in law to  
\begin{multline*}
\max_{0=t_0\leq t_1\leq\dots\leq t_{k^*}=1}\sqrt{p_{\max}}\Bigg[ \left(\frac{\sqrt{1-k^*p_{\max}}-1}{k^*}\sum_{i=1}^{k^*} \overline{B}^X_i(1)+\sum_{i=1}^{k^*} \left(\overline{B}^X_i(t_i)-\overline{B}^X_i(t_{i-1})\right)\right)\wedge \\ \left(\frac{\sqrt{1-k^*p_{\max}}-1}{k^*}\sum_{i=1}^{k^*} \overline{B}^Y_i(1)+\sum_{i=1}^{k^*} \left(\overline{B}^Y_i(t_i)-\overline{B}^Y_i(t_{i-1})\right)\right)\Bigg],
\end{multline*}
where now $\overline{B}^X$ and $\overline{B}^Y$ are two independent 
$k^*$-dimensional standard Brownian motions on $[0,1]$. 
Dividing both sides by $\sqrt{p_{\max}}$, one 
obtains the conjectured Theorem~\ref{samedistthm} which reduces to 
Theorem~\ref{unfthm} when 
$k^*=m$.

\subsection{Generalization to any fixed sequence of blocks}\label{genblock}
As pointed out by an Associate Editor, and also developed, 
for binary alphabets, in \cite{YuZ}, a longest common increasing
subsequence can be viewed as a longest common subsequence where letters
are aligned in blocks. (For $LCI_n$, a non-void block only aligns a single type of letter 
and the first block consists of the letter $\alpha(1):=1$, then the second one consists of $\alpha(2):=2$ and so on, up to the last block eventually consisting of the letter $\alpha(m):=m$.)  So, more generally, one could investigate the longest common subsequences where letters are aligned in blocks of letters 
$\alpha(1),\dots,\alpha(l)$, for any $l\geq m$, and where $\alpha:\ens{l}\rightarrow\mathcal{A}_m$ is onto.  For any fixed $\alpha$, the length of the longest common subsequences where letters are aligned with blocks $\alpha$ is at most equal to $LC_n$, the length of the longest common subsequences, 
and moreover, $LC_n$ is the maximum of these lengths over all the possible block-orders $\alpha$ 
($l$ is not fixed). To pass 
from the block version to $LC_n$, there is, however, a major issue of interversion of limits.   In what 
follows, at first, we merely give for any fixed $\alpha$, the limiting law of the length of the (rescaled) longest common subsequences where letters are aligned in blocks $\alpha(1),\dots,\alpha(l)$, and then 
the corresponding limiting laws, when allowing 
for a fixed numbers of such blocks.

Firstly, defining for any $k\in\mathbb{N}$, $k\ge 2$, $\Lambda_k:=\{\lambda\in\left(\mathbb{R}_+\right)^{k}=\::\: \lambda_1+\dots+\lambda_k=1\}$, we claim that: 

\begin{equation}\label{egmax}
\max_{\substack {\lambda\in\Lambda_l^2}} \sum_{i=1}^l \left[\left(p^X_{\alpha(i)} \lambda^X_{\alpha(i)}\right)\wedge \left(p^Y_{\alpha(i)} \lambda^Y_{\alpha(i)}\right)\right]=\max_{\substack{\lambda\in\Lambda_m^2}} \sum_{i=1}^m \left[\left(p^X_i \lambda^X_i \right)
\wedge \left(p^Y_i \lambda^Y_i\right)\right].
\end{equation}
Indeed to see the validity of this equality, note that above the left-hand side is greater or equal than the right-hand side since $\alpha$ is onto, while it is also less or equal since we can partition $\ens{l}$ via 
$\alpha^{-1}(\{1\}), \alpha^{-1}(\{2\}), \dots, \alpha^{-1}(\{m\})$ and use the basic inequality 
$(a\wedge b) + (c\wedge d)\leq (a+c)\wedge(b+d)$.  

Next, to adapt the proof of our main theorem, we need to define the set $U^\alpha$, as well as all other quantities which depended on $m$ or $p$, with $l$ instead of $m$ and 
$p^\bullet_{\alpha(1)},\dots,p^\bullet_{\alpha(l)}$ instead of $p^\bullet_1,\dots,p^\bullet_m$.  
Note also that, when $l>m$, the quantities $p^\bullet_{\alpha(1)},\dots,p^\bullet_{\alpha(l)}$ do not form a 
probability mass function (their sum is not equal to one), but all their elements are positive 
which is enough to have everything well defined.

Formally, for example, 
$$U^\alpha:=\left\{u\in\mathbb{R}_+^l\::\: \frac{u_1}{p^X_{\alpha(1)}}+\dots+\frac{u_l}{p^X_{\alpha(l)}}\leq 1, \frac{u_1}{p^Y_{\alpha(1)}}+\dots+\frac{u_l}{p^Y_{\alpha(l)}}\leq 
1\right\},$$
$\phi^\alpha:\mathbb{R}^{l}\rightarrow \mathbb{R}$ is given by 
\begin{equation*}
\phi^\alpha:u\mapsto u_1+\dots+u_l,
\end{equation*}
and $I^\alpha$ is now defined to be the set of integers $i\in\ens{l}$ such that there exists 
$u^i\in L_{U^\alpha}$ with $u^i>0$. Using almost the same proof as the one showing the 
equality of the two maxima in \eqref{egmax}, we 
get $\alpha^{-1}(I)=I^\alpha$, where $I$ is defined as 
before. There is no need to redefine the 
various cases a), b1), b2) here since they coincide with those  
previously defined  when taking $p^\bullet_{\alpha(1)},\dots,p^\bullet_{\alpha(l)}$ instead of 
$p^\bullet_1,\dots,p^\bullet_m$.  For example, "there exists $u\in U^\alpha$ maximizing 
$\phi^\alpha$ over $U^\alpha$ such 
that $ \frac{u_1}{p^X_{\alpha(1)}}+\dots+\frac{u_l}{p^X_{\alpha(l)}}=1$ and $\frac{u_1}{p^Y_{\alpha(1)}}+\dots+\frac{u_l}{p^Y_{\alpha(l)}}<1$" is equivalent to Case a) defined in 
Section~\ref{secamean}.
Finally, the function $\mathfrak{m}$ defined in Lemma~\ref{exprm} can be extended naturally to $\left(\mathbb{R}^l\right)^2$.

Within this generalized setting, the proof of Lemma~\ref{lemme} carries over, giving us the 
following theorem for, $LC^\alpha_n$, the length of the longest common 
subsequences with blocks $\alpha(1),\dots,\alpha(l)$.

\begin{thm}\label{bloc}
Let $B^X$ and $B^Y$ be two independent $|I|$-dimensional Brownian motions defined on 
$[0,1]$ with respective covariance matrix $C^X$ defined by $C^X_{i,i}=1$ and $C^X_{i,j}=-\sqrt{\frac{p^X_{\alpha(i)} p^X_{\alpha(j)}}{(1-p^X_{\alpha(i)})(1-p^X_{\alpha(j)})}}$, for $i\neq j$ in $I$, and $C^Y$ defined in a similar fashion. For all $\lambda\in   K_{\Lambda^2}^\alpha$ and $i\in I^\alpha$, set \begin{align*}
V^{\alpha,X}_i(\lambda^X)=\sqrt{p^X_{\alpha(i)}(1-p^X_{\alpha(i)})}\left(B^X_{\alpha(i)}\left(\sum_{j=1}^i\lambda^X_j\right)-B^X_{\alpha(i)}\left(\sum_{j=1}^{i-1}\lambda^X_j\right)\right),\\
V^{\alpha,Y}_i(\lambda^Y)=\sqrt{p^Y_{\alpha(i)}(1-p^Y_{\alpha(i)})}\left(B^Y_{\alpha(i)}\left(\sum_{j=1}^i\lambda^Y_j\right)-B^Y_{\alpha(i)}\left(\sum_{j=1}^{i-1}\lambda^Y_j\right)\right).
\end{align*}
If there exists $u\in L_{U^{\alpha}}$ such that $ \frac{u_1}{p^X_{\alpha(1)}}+\dots+\frac{u_l}{p^X_{\alpha(l)}}=1$ 
and $\frac{u_1}{p^Y_{\alpha(1)}}+\dots+\frac{u_l}{p^Y_{\alpha(l)}}<1$, or equivalently if there 
exists $u\in L_{U}$ such that $ \frac{u_1}{p^X_1}+\dots+\frac{u_m}{p^X_1}=1$ 
and $\frac{u_1}{p^Y_1}+\dots+\frac{u_m}{p^Y_m}<1$ (Case a)), then  
\begin{equation}
\frac{LC_n^\alpha-ne_{\max}}{\sqrt{n}}\xRightarrow[n \to \infty]{}Z^a:
=\max_{\substack {\lambda^X\in J^\alpha}} \sum_{i\in I^\alpha} V^{\alpha,X}_i(\lambda^X).
\end{equation}
If for all $u\in L_{U^{\alpha}}$, $ \frac{u_1}{p^X_{\alpha(1)}}+\dots+\frac{u_l}{p^X_{\alpha(l)}}=1$ and $\frac{u_1}{p^Y_{\alpha(1)}}+\dots+\frac{u_l}{p^Y_{\alpha(l)}}=1$, or equivalently if 
for all $u\in L_{U}$, $ \frac{u_1}{p^X_1}+\dots+\frac{u_m}{p^X_1}=1$ 
and $\frac{u_1}{p^Y_1}+\dots+\frac{u_m}{p^Y_m}=1$ (Case b)), then 
\begin{equation}
\frac{LC_n^\alpha-ne_{\max}}{\sqrt{n}}\xRightarrow[n \to \infty]{}Z^b:=\max_{\substack {\lambda\in   K_{\Lambda^2}^\alpha}} \mathfrak{m}\left(V^{\alpha,X}(\lambda^X),V^{\alpha,Y}(\lambda^Y)\right),
\end{equation}
where, again, now $\mathfrak{m}$ is defined on $\left(\mathbb{R}^l\right)^2$.

\end{thm}

For instance, for $m=2$ and in the uniform case, the order 
$\alpha(1)=2,\alpha(2)=1,\alpha(3)=2$ gives the limiting distribution: 

\begin{equation}
\frac{LC_n^\alpha-ne_{\max}}{\sqrt{n}}\xRightarrow[n \to \infty]{}Z^b:=\max_{\substack {\lambda^X_1+\lambda^X_2
+\lambda^X_3=1 \\ \lambda^Y_1+\lambda^Y_2+\lambda^Y_3=1}} 
\mathfrak{m}\left(V^{\alpha,X}(\lambda^X),V^{\alpha,Y}(\lambda^Y)\right),
\end{equation}
i.e., 
\begin{equation}
\frac{LC_n^\alpha-ne_{\max}}{\sqrt{n}}\xRightarrow[n \to \infty]{}Z^b:=\frac{1}{2}\max_{\substack {\lambda^X_1+\lambda^X_2+\lambda^X_3=1 \\ \lambda^Y_1+\lambda^Y_2+\lambda^Y_3=1}} 
 \min_{\bullet\in\{X,Y\}} \left(B^\bullet_{2}(\lambda^\bullet_1)+B^\bullet_{1}(\lambda^\bullet_1+\lambda^\bullet_2)-B^\bullet_{1}(\lambda^\bullet_1)+B^\bullet_{2}(1)-B^\bullet_{2}(\lambda^\bullet_1+\lambda^\bullet_2)\right).  
\end{equation}

Also note that, sometimes, the limit in the above theorem 
is simply a normal random variable.   
Indeed, take $p^X_1=1/3,p^X_2=2/3,p^Y_1=1/4,p^Y_2=3/4$, and 
$\alpha(1)=1, \alpha(2)=2$, then we 
are in Case a), $I=\{2\}$ and:

\begin{equation}
\frac{LC_n^\alpha-ne_{\max}}{\sqrt{n}}\xRightarrow[n \to \infty]{}Z^a:=\frac{\sqrt{2}}{3}B^X_2(1).
\end{equation}

This is also, as one would expect, the limiting distribution of the number of 2's in the first word (which is almost equal to $LC_n^\alpha$). However, if we take $\alpha(1)=2,\alpha(2)=1,\alpha(3)=2$, 
the limit is more involved.

For $b\in\mathbb{N}$ such that $b\geq m$, let now $F_m^b$ denote the set of all 
surjections from  $\ens{b}$ to $\ens{m}$, and let $LC_n^{(b)}$ be the length of the longest common subsequences with $b\geq m$ blocks, with for each letter at least one block of this letter, and still allowing the blocks to have size zero. This is nothing but the maximum, over all the possible $\alpha\in F_m^b$, of $LC_n^\alpha$, so, recalling the discussion preceding the statement of Theorem~\ref{bloc}, we have:

\begin{thm}
In Case a),  \begin{equation}
\frac{LC_n^{(b)}-ne_{\max}}{\sqrt{n}}\xRightarrow[n \to \infty]{}Z^a:=\max_{\substack {\lambda^X\in J^\alpha \\ \alpha\in F_m^{b}}} \sum_{i\in I^\alpha} V^{\alpha,X}_i(\lambda^X).
\end{equation}
In Case b), \begin{equation}
\frac{LC_n^{(b)}-ne_{\max}}{\sqrt{n}}\xRightarrow[n \to \infty]{}Z^b:=\max_{\substack {\lambda\in   K_{\Lambda^2}^\alpha\\ \alpha\in F_m^{b}}} \mathfrak{m}\left(V^{\alpha,X}(\lambda^X),V^{\alpha,Y}(\lambda^Y)\right).  
\end{equation}
\end{thm}
\begin{proof}
The proof of this theorem follows lines of 
the proof of our previous main result, considering 
$p^\bullet_{\alpha(i)}$ instead of $p^\bullet_i$.
\end{proof}

Note that $LC_n$, the length of the longest common subsequences without any conditions on blocks, corresponds to $LC_n^{(n+m)}$ (or to be more precise, $LC_n^{(b)}$ for any $b\geq m+n-2$: this is because when, say, there are only two kind of letters involved in the longest common word, we have to take $m-2$ additional empty blocks to make $\alpha$ onto).  
Although the above theorem requires a fixed number of blocks, say, $b$, it is nevertheless 
noteworthy that no matter this fixed number, 
$$\lim_{n\to+\infty}\frac{\mathbb{E}LC_n^{(b)}}{n}=e_{\max}.$$

\subsection{Countably infinite alphabet}
To continue, let us consider, as in \cite[Section 4]{HL1}, the generalization to 
countably infinite alphabets. Let the alphabet be $\mathbb{N}^*=\{1,2,\dots\}$, 
let $(p'^X_i)_{i\geq 1}$ and $(p'^Y_i)_{i\geq 1}$ be two probability mass functions on this 
alphabet, we are now 
interested in $LCI^{\infty}_n$, the length of the longest common and increasing subsequences 
over this countably infinite alphabet.  
Let \begin{equation*}
\Lambda^{\infty}=\left\{\lambda\in\left(\mathbb{R}_+\right)^{\mathbb{N}^*}=[0,+\infty)^{\mathbb{N}^*}\::\: \sum_{i=1}^{+\infty}\lambda_i=1\right\},
\end{equation*} and let 
\begin{equation}
e_{\max}^{\infty}=\sup_{\substack{\lambda\in(\Lambda^\infty)^2}} \sum_{i=1}^{+\infty} \left[\left(p'^X_i\lambda^X_i\right)\wedge \left(p'^Y_i\lambda^Y_i\right)\right].
\end{equation}
Let $m\in\mathbb{N}, m\geq 2$ be such that $\sum_{i=m}^{+\infty} p'^X_i < e_{\max}^{\infty}$ 
and $\sum_{i=m}^{+\infty} p'^Y_i <e_{\max}^{\infty}$. Let us consider the distributions over $\{1,\dots,m\}$ obtained by replacing all the letters greater or equal to $m$ by $m$, namely, let $p^X_i=p'^X_i$ for $i<m$ and $p^X_m:=\sum_{i=m}^{+\infty} p'^X_i$, and let $p^Y_i$, $1\leq i\leq m$, be defined in a similar fashion.  Let now $LCI_n$ be the length 
of the longest increasing subsequences formed by 
replacing all the letters greater or equal to $m$ by $m$, i.e., the  longest 
common and increasing 
subsequences on $\{1, \dots, m\}$ associated with the probability mass functions $p'^X$ 
and $p'^Y$.  Next we argue, via a sandwiching argument,  
that when properly centered and scaled (note that $e_{\max}^{\infty}=e_{\max}$), 
$LCI^{\infty}_n$ and $LCI_n$  tend to the same limit. 
Indeed, let $LCI^*_n$ be 
the length of the longest common and  increasing subsequences not using the letter $m$, 
i.e., the length of the longest common and increasing subsequences on $\{1, \dots, m-1\}$ associated with the probability mass functions $p'^X$ 
and $p'^Y$ or, equivalently,  
$p^X$ and $p^Y$.   
Since $m\notin I$ (where $I$ is defined with the distribution 
$(p^X_i)_{1\leq i\leq m}$ and $(p^Y_i)_{1\leq i\leq m}$), 
$(LCI^*_n - ne_{\max})/\sqrt{n}$ and $(LCI_n - ne_{\max})/\sqrt{n}$ converge to the same limiting distribution. 
But, 
\begin{equation*}
\frac{LCI^*_n - ne_{\max}}{\sqrt{n}}\leq \frac{LCI^{\infty}_n - ne_{\max}}{\sqrt{n}}\leq \frac{LCI_n - ne_{\max}}{\sqrt{n}},
\end{equation*} completing the proof.

From the proofs presented above, the passage from two to three or more sequences is
clear: the minimum over two Brownian functionals becomes a minimum over three or more
Brownian functionals, and such a passage applies to the cases touched upon above and below.  

Throughout the text, the two sequences $(X_k)_{k\ge 1}$ and $(Y_k)_{k\ge 1}$ are assumed to be independent 
with respective i.i.d.~components.  In view of \cite{HL2} or \cite{HK1}, one expects that the i.i.d.~assumption 
could be replaced by a Markovian one or even a hidden Markovian one.  
Moreover, one further expects that the independence of the two sequences 
is unnecessary and that a potential dependence structure 
between the two sequences would carry over to corresponding $2m$-dimensional Brownian functionals, 
another case at hand could be the hidden Markov framework.   Finally, it should also be of interest 
(as already done in \cite{breton2017limiting} for uniform letters) to study the ramifications/connections 
of our results with last passage percolation.

\section*{Appendix: proof of Lemma~\ref{exprm}}
\begin{proof}
Define $f_\nu:E'^2\rightarrow \mathbb{R}$ by 
$f_\nu:x\mapsto \sum_{i=1}^m \left[\left(p^X_i x^X_i
+\nu^X_i\right)\wedge \left(p^Y_i x^Y_i+\nu^Y_i\right)\right]$. In order to prove that $\mathfrak{m}(\nu)$ is well defined and \eqref{local}, it is enough to prove that for all $x\in E'^2$, there exists $x'\in E'^2$ such that $\|x'\|_{\infty}\leq 2Cm \|\nu\|_{\infty}$ and $f_\nu (x')\geq f_\nu (x)$. 
Let $x\in E'^2$.  Firstly, assume that $x\in P$ 
(recalling \eqref{defP}). If $f_\nu(x)<f_\nu(0)$, taking $x'=0$ works, so assume 
$f_\nu(x)\geq f_\nu(0)$. By \eqref{elem} (applied twice), \begin{equation*}
-m\|\nu\|_\infty \leq f_\nu(0)\leq f_\nu(x)\leq m\|\nu\|_\infty+f(x)
\end{equation*}
hence $-f(x)\leq 2m\|\nu\|_\infty$ and, by Lemma \ref{proj}, there exists $x^{K\cap P}\in K\cap P$ and $x^r\in E^2$ such that $x=x^{K\cap P}+x^r$ and $\|x^r\|_\infty\leq -Cf(x)\leq 2Cm\|\nu\|_\infty$. 
But from the definition of $K$, $f_\nu(x^{K\cap P}+x^r)=f(x^{K\cap P})+f_\nu(x^r)$, and by $\eqref{propkp}$, $f(x^{K\cap P})=0$ so $f_\nu(x)=f_\nu(x^r)$. Moreover, since $x\in P$ and 
$x^{K\cap P,\bullet}_i=0$ for all $i\in I^c$, $x^r\in E'^2$.

Now, if we do not assume $x\in P$ anymore, observe that for $\varepsilon>0$ small enough, 
$\varepsilon x\in P$, so $f_{\varepsilon\nu}(x')\geq f_{\varepsilon\nu}(\varepsilon x)$ for some $x'\in E'^2$ such that $\|x'\|_{\infty}\leq 2Cm \|\varepsilon\nu\|_{\infty}$. Finally, dividing by $\varepsilon$, 
$f_\nu ((1/\varepsilon)x')\geq f_\nu(x)$ where  $\|(1/\varepsilon)x'\|_{\infty}\leq 2Cm \|\nu\|_{\infty}$.

In Case b1), let us begin with the subcase $I=\{1\}$.   In this instance, $p^X_1=p^Y_1=e_{\max}$, while for all $1<i\leq m$, $p^X_i< e_{\max}$ or $p^Y_i<e_{\max}$ (otherwise $i$ would be in $I$).  
We now show that ``the maximum of $f_\nu$ is realized with the first letter plus one other letter'', more precisely, there exists $x\in E'^2$ such that $f_\nu(x)=\mathfrak{m}(\nu)$ and $|\{i\in\{2,\dots,m\}: x^X_i\neq 0 \text{ or } x^Y_i\neq 0\}|\leq 1$. Indeed, using the same method than in the proof of Lemma~\ref{reductotwo}, keeping in mind $\nu^\bullet_2=\dots=\nu^\bullet_m=0$, one can see that there exists some $x$ maximizing $f_\nu$ such that $\{i\in\{1,\dots,m\}: x^X_i\neq 0 \text{ or } 
x^Y_i\neq 0\}$ has at most two elements, and they can't both belong to $\{2,\dots,m\}$ otherwise they would be null (by the definition of $E'$).

Returning to the proof of the lemma, we have shown that 
\begin{equation}
\max_{x\in E'^2}f_\nu(x)=\max_{i_0\in\{2,\dots,m\}}\sup_{\substack{x\in E'^2\\ 
\forall i\in\{2,\dots,m\}\setminus\{i_0\}, x^\bullet_i=0}}f_\nu(x).
\end{equation}
Fixing $i_0\in\{2,\dots,m\}$, we have 
\begin{equation}
\sup_{\substack{x\in E'^2\\ \forall i\in\{2,\dots,m\}\setminus\{i_0\}, x^\bullet_i=0}}f_\nu(x)=\sup_{t^X,t^Y>0}\left[(\nu^X_1-e_{\max}t^X)\wedge(\nu^Y_1-e_{\max}t^Y)+(p^X_{i_0}t^X)\wedge(p^Y_{i_0}t^Y)\right].  
\end{equation} 

It is then easily seen that this last supremum does not change with the additional condition $p^X_{i_0}t^X=p^Y_{i_0}t^Y$.  (Indeed, if, for example, $p^X_{i_0}t^X>p^Y_{i_0}t^Y$, reducing $t^X$ to transform 
this strict inequality into equality will only increase the sum of the two minima 
in the definition of $f_\nu$.)  Hence, 
\begin{align*}
\sup_{\substack{x\in E'^2\\ \forall i\in\{2,\dots,m\}\setminus\{i_0\}, x^\bullet_i=0}}f_\nu(x)&=\sup_{t^X>0}\left[\left(\nu^X_1-e_{\max}t^X+p^X_{i_0}t^X\right)\wedge\left(\nu^Y_1-e_{\max}\frac{p^X_{i_0}}{p^Y_{i_0}}t^X+p^X_{i_0}t^X\right)\right]\\
&=\sup_{t^X>0}\left[\left(\nu^X_1+(p^X_{i_0}-e_{\max})t^X\right)\wedge\left(\nu^Y_1+\frac{p^X_{i_0}}{p^Y_{i_0}}(p^Y_{i_0}-e_{\max})t^X\right)\right].
\end{align*}

Since $i_0\notin I$, it is impossible for both $p^X_{i_0}-e_{\max}$ and $p^Y_{i_0}-e_{\max}$ to be positive, so this last supremum is attained at $t^X=0$ (and is equal to $\nu^X_1\wedge\nu^Y_1$) 
unless 
$\nu^X_1<\nu^Y_1$ and $p^X_{i_0}-e_{\max}>0$, or 
$\nu^X_1>\nu^Y_1$ and $p^Y_{i_0}-e_{\max}>0$, in which case the supremum  is attained at 
$t^X=\frac{p^Y_{i_0}}{e_{\max}}\frac{\nu^Y_1-\nu^X_1}{p^X_{i_0}-p^Y_{i_0}}$, 
a value at which the two sides in the above minimum are equal 
to each other. So if $\nu^X_1<\nu^Y_1$ and $p^X_{i_0}-e_{\max}>0$, or $\nu^X_1>\nu^Y_1$ 
and $p^Y_{i_0}-e_{\max}>0$, then
\begin{equation}
\sup_{\substack{x\in E'^2\\ \forall i\in\{2,\dots,m\}\setminus\{i_0\}, x^\bullet_i=0}}f_\nu(x)=\frac{p^X_{i_0}(e_{\max}-p^Y_{i_0})}{e_{\max}(p^X_{i_0}-p^Y_{i_0})}\nu^X_1+\frac{p^Y_{i_0}(p^X_{i_0}-e_{\max})}{e_{\max}(p^X_{i_0}-p^Y_{i_0})}\nu^Y_1.
\end{equation}
Assuming that $\nu^X_1<\nu^Y_1$, we see that in this case 
$\mathfrak{m}(\nu^X,\nu^Y)=s_X S^Y+t_X S^X$.  This remains true if $S^X=S^Y$ (in this case, 
$\mathfrak{m}(\nu^X,\nu^Y)=S^X=S^Y$), and, similarly, when $S^Y\leq S^X$. 
The proof of Case b1) is therefore done when $I=\{1\}$.

Still in Case b1), but without the assumption that $I=\{1\}$, assume, without loss of generality, 
that $I=\{1,\dots,k\}$, $k\ge 2$.  
Define $\tilde{\nu}$ by $\tilde{\nu}^\bullet_1=S^\bullet$ and $\tilde{\nu}^\bullet_i=0$, for all 
$i\geq 2$. Let $x^0\in E'^2$ be defined by $x^{0,Y}=0$, $x^{0,X}_1=(S^X-S^Y+\nu^Y_1-\nu^X_1)/{e_{\max}}$, $x^{0,X}_i=(\nu^Y_i-\nu^X_i)/{e_{\max}}$, for all $i\in\{2,\dots,k\}$, and $x^{0,\bullet}_i=0$ for all $i\in\{k+1,\dots,m\}$. Note that for all $x\in E'^2$, $f_\nu(x+x^0)=f_{\tilde{\nu}}(x)$, 
so $\mathfrak{m}(\nu)=\mathfrak{m}(\tilde{\nu})$.  Moreover, defining $x'$ via $x'^\bullet_1=x^\bullet_1+\dots+x^\bullet_k$, $x'^\bullet_i=0$, for $i\in\{2,\dots,k\}$, and $x'^\bullet_i=x^\bullet_i$ everywhere else, we have $x'\in E'^2$, and \begin{multline}
\left(e_{\max}(x^X_1+\dots+x^X_k)+\tilde{\nu}^X_1\right)\wedge\left(e_{\max}(x^Y_1+\dots+x^Y_k)+\tilde{\nu}^Y_1\right)\geq (e_{\max}x^X_1+\tilde{\nu}^X_1)\wedge(e_{\max}x^Y_1+\tilde{\nu}^Y_1)\\
+e_{\max}(x^X_2+\dots+x^X_k)\wedge (x^Y_2+\dots+x^Y_k),  
\end{multline}
\begin{multline}
\left(e_{\max}(x^X_1+\dots+x^X_k)+\tilde{\nu}^X_1\right)\wedge\left(e_{\max}(x^Y_1+\dots+x^Y_k)+\tilde{\nu}^Y_1\right)\geq (e_{\max}x^X_1+\tilde{\nu}^X_1)\wedge(e_{\max}x^Y_1+\tilde{\nu}^Y_1)\\
+e_{\max}(x^X_2\wedge x^Y_2)+\dots+(x^X_k\wedge x^Y_k).  
\end{multline}
Hence, $f_{\tilde{\nu}}(x')\geq f_{\tilde{\nu}}(x)$, and therefore  
\begin{equation}
\mathfrak{m}(\tilde{\nu})=
\max_{\substack{x\in E'^2\\ \forall i\in\{2,\dots,k\}, x^\bullet_i=0}}f_{\tilde{\nu}}(x).  
\end{equation}
Now applying the subcase $I=\{1\}$ concludes the proof of Case b1).

In Case b2),  again assume without loss of generality that $I=\{1,\dots,k\}$, $k\ge 2$.  
Let $L_1=(1,0,\dots,0,-1,0,\dots,0)\in\mathbb{R}^{2k}$, having 
$k-1$ zeros between the two non-zero coordinates, 
let $L_2=(0,1,0,\dots,0,-1,0,\dots,0)$ (still with $k-1$ zeros between 
the two non-zero coordinates), and iterate this process up to $L_{k}$.  
Let also $\widetilde{P^X}$ be the concatenation of $P^X\in\mathbb{R}^{k}$ with  
$0\in\mathbb{R}^{k}$, and let $\widetilde{P^Y}$ be the concatenation 
of $0\in\mathbb{R}^{k}$ with $P^Y\in\mathbb{R}^{k}$. 
The vectors $L_1,\dots,L_{k},\widetilde{P^X},\widetilde{P^Y}$ are linearly independent since, 
as already seen in Lemma~\ref{indep}, $P^X$ and $P^Y$ are linearly independent. 
Now, let $Q$ be a $2k\times 2k$ invertible matrix with first rows 
$L_1,\dots,L_{k},\widetilde{P^X},\widetilde{P^Y}$ (for example, to form 
such a matrix $Q$, one could complete the first columns with vectors from 
the canonical basis), let $\Delta\in\mathbb{R}^{2k}$ be defined by 
\begin{equation}
\Delta_i:=\begin{cases} 
  \nu^Y_i-\nu^X_i & 
  \text{if } i\in\ens{k} \\
      0, & \text{if } i\in\{k+1,\dots,2k\}, \\
      \end{cases}
\end{equation} and let $u\in\mathbb{R}^{2k}$ be defined by\begin{equation}
u_i:=\begin{cases} 
  (Q^{-1}\Delta)_i & 
  \text{if } i\in\ens{k} \\
      0, & \text{if } i\in\{k+1,\dots,2k\}. \\
      \end{cases}
\end{equation}
We have $u^X_i-u^Y_i=\nu^Y_i-\nu^X_i$ (where $u^X$ is the vector of the first $k$ coordinates 
of $u$ and $u^Y$ the vector of the last $k$ coordinates of $u$) for all $i\in\ens{k}$ : these conditions stem from the rows $L_1,\dots,L_{k}$.  Moreover, $u^X_1/p^X_1+\dots+u^X_m/p^X_m=u^Y_1/p^Y_1+\dots+u^Y_m/p^Y_m=0$ (conditions stemming from the rows $\widetilde{P^X},\widetilde{P^Y}$).  Then,  expand $u^X$ and $u^Y$ to $\mathbb{R}^m$ by filling with zeros, so that $u:=(u^X,u^Y)$ is now in $\left(\mathbb{R}^{m}\right)^2$.  
Setting, for all $i\in \ens{m}$, 
$y^X_i:=u^X_i/p^X_i, y^Y_i:=u^Y_i/p^Y_i$, lead to $y\in \left(\mathbb{R}^m\right)^2$, more precisely $y\in E'^2$ such that for all $i\in \ens{m}, p^X_i y^X_i+\nu^X_i=p^Y_i y^Y_i+\nu^Y_i$, 
with moreover 
\begin{align*}
\sum_{i=1}^m \left[\left(p^X_i y^X_i+\nu^X_i\right)\wedge \left(p^Y_i y^Y_i+\nu^Y_i\right)\right]=\sum_{i\in I} \left(\frac{p^X_i y^X_i+\nu^X_i+p^Y_i y^Y_i+\nu^Y_i}{2}\right).  
\end{align*}
Setting  $U^X:=(u^X_i)_{i\in I}\in \mathbb{R}^k$, $U^Y:=(u^Y_i)_{i\in I}$, $R^X:=(\nu^X_i)_{i\in I}$ and $R^Y:=(\nu^Y_i)_{i\in I}$, 
the above expression becomes  
\begin{align*}
\sum_{i=1}^m \left[\left(p^X_i y^X_i+\nu^X_i\right)\wedge \left(p^Y_i y^Y_i+\nu^Y_i\right)\right]&
=\frac{1}{2}(U^X+R^X+U^Y+R^Y)\cdot (1)_{i\in I}.  
\end{align*}

With the notations of Lemma~\ref{indep},
\begin{align*}
U^X\cdot (1)_{i\in I}&=U^X\cdot(sP^X+tP^Y)\\
&=U^X\cdot tP^Y\\
&=(U^X-U^Y)\cdot tP^Y\\
&=(R^Y-R^X)\cdot tP^Y.  
\end{align*}
Similarly, $U^Y\cdot (1)_{i\in I}=(R^X-R^Y)\cdot sP^X$. So, 
\begin{align*}
\sum_{i=1}^m \left[\left(p^X_i y^X_i+\nu^X_i\right)\wedge \left(p^Y_i y^Y_i+\nu^Y_i\right)\right]&=\frac{1}{2}(R^X-R^Y)\cdot(sP^X-tP^Y)+\frac{1}{2}(R^X+R^Y)\cdot(sP^X+tP^Y)\\
&=R^X\cdot sP^X+R^Y\cdot tP^Y\\
&=\sum_{i\in I} \left(\frac{s}{p^X_i}\nu^X_i+ \frac{t}{p^Y_i}\nu^Y_i\right).
\end{align*}
This shows that $\max_{x\in E'^2} \sum_{i=1}^m \left[\left(p^X_i x^X_i
+\nu^X_i\right)\wedge \left(p^Y_i x^Y_i+\nu^Y_i\right)\right] \geq \sum_{i\in I}\left({s\nu^X_i}/{p^X_i}+ 
{t\nu^Y_i}/{p^Y_i}\right)$. Now let $x\in E'^2$, 
\begin{align*}
&\sum_{i=1}^m\! \left[\left(p^X_ix^X_i+\nu^X_i\!\right)\wedge \left(p^Y_ix^Y_i+\nu^Y_i\!\right)\right]
\!\!-\sum_{i\in I}\! \left(\!
\frac{s}{p^X_i}\nu^X_i+ \frac{t}{p^Y_i}\nu^Y_i\!\!\right)\\&\qquad\qquad=\sum_{i=1}^m \!\left[\left(p^X_ix^X_i\!+\nu^X_i\right)\!\wedge \left(p^Y_ix^Y_i\!+\nu^Y_i\!\right)\right]\!\!-\sum_{i=1}^m \left[\left(p^X_i y^X_i+\nu^X_i\right)\wedge \left(p^Y_i y^Y_i+\nu^Y_i\right)\right]\\
&\qquad\qquad=\sum_{i=1}^m \left[\left(p^X_i (x-y)^X_i\right)\wedge \left(p^Y_i (x-y)^Y_i\right)\right]\\
&\qquad\qquad=f(x-y).
\end{align*}
We have $x-y\in E'^2$ (recall, also, that $y_i=0$ for all $i\in I^c$), so for some $c>0$, $(x-y)/c
\in P$, and then $f((x-y)/c)\leq 0$, so $f(x-y)\leq 0$.  
Hence  $\sum_{i=1}^m \left[\left(p^X_i x^X_i+\nu^X_i\right)\wedge \left(p^Y_i x^Y_i+\nu^Y_i\right)\right] -\sum_{i\in I} \left({s\nu^X_i/p^X_i}+ {t\nu^Y_i/p^Y_i}\right)\leq 0$ and, finally,  
$\max_{x\in E'^2} \sum_{i=1}^m \left[\left(p^X_i x^X_i+\nu^X_i\right)\wedge \left(p^Y_i x^Y_i+\nu^Y_i\right)\right]=\sum_{i\in I} \left({s\nu^X_i/p^X_i}+ {t\nu^Y_i/p^Y_i}\right)$.
\end{proof}


\begin{thebibliography}{99}

\bibitem{B-GH} F.~Benaych-Georges and C.~Houdr\'e.
\newblock GUE minors, maximal Brownian functionals and longest increasing subsequences in random words.  
{\it Markov Processes. Related Fields} {\bf 21} (2015), 109-126.

\bibitem{breton2017limiting}
J.-C.~Breton and C.~Houdr\'e.
\newblock On the limiting law of the length of the longest common and increasing subsequences in random words. {\it Stochastic Process. Appl.} {\bf 127} (2017), 1676–1720.
  
 \bibitem{HK1}  C.~Houdr\'e and G.~Kerchev.
 \newblock   On the rate of convergence for the length of the longest common subsequences in hidden Markov models. 
 {\it J. Appl. Probab.} {\bf 56} (2019), no. 2, 558–573
  
 \bibitem{HLM}  C.~Houdr\'e, J.~Lember and H.~Matzinger. 
 \newblock On the longest common increasing binary subsequence. 
 {\it C.R. Acad. Sci., Paris Ser. I} {\bf 343} (2006), 589--594.  
  
\bibitem{HL1}
C.~Houdr\'e and T.~J.~Litherland.
\newblock On the longest increasing subsequence for finite and countable alphabets. {\it High Dimensional Probability V: The Luminy Volume} (2009), 185-212.

\bibitem{HL2}
C.~Houdr\'e and T.~J.~Litherland.
\newblock On the limiting shape of Young diagrams associated with Markov random words.  
{\it Markov Processes. Related Fields} {\bf 26} (2020), 779-838.  


\bibitem{KLM}  M.~Kiwi, M.~Loebl and J.~Matoušek.  
\newblock Expected length of the longest common subsequence for large alphabets. 
{\it Adv. Math.} {\bf 197} (2005), 480–498. 



\bibitem{YuZ}
Y.~Zhang
\newblock Topics on the length of the longest common subsequences with blocks in binary random words.  
{\it PhD dissertation, Georgia Institute of Technology} (2019).  




\end{thebibliography}
\end{document}